\documentclass[11pt,A4wide]{article}

\usepackage{graphicx,psfrag,pstricks,pslatex}
\usepackage[leqno]{amsmath}
\usepackage{amssymb,amscd,amsthm,amsfonts}
\usepackage{array}
\usepackage[francais,english]{babel}
\usepackage[latin1]{inputenc}
\usepackage[T1]{fontenc}
\usepackage{mathrsfs}

\newtheoremstyle{dteo}
     {2pt}
     {0pt}
     {\slshape}
     {}
     {\bfseries}
     {:}
     {.5em}
     {}
     
\newtheoremstyle{drem}
     {2pt}
     {3pt}
     {\rmfamily}
     {}
     {\itshape}
     {:}
     {.5em}
     {}

\newtheoremstyle{ddef}
     {2pt}
     {3pt}
     {\rmfamily}
     {}
     {\bfseries}
     {:}
     {.5em}
     {}

\newtheoremstyle{dqst}
     {1pt}
     {1pt}
     {\itshape}
     {}
     {\bfseries}
     {!?!}
     {.5em}
     {}

\theoremstyle{dteo}
\newtheorem{teo}{Th\'eor\`eme}[subsection]

\newtheorem{prop}[teo]{Proposition}

\theoremstyle{drem}

\theoremstyle{ddef}

\theoremstyle{dqst}

\theoremstyle{dteo}
\newtheorem{teoa}[teo]{Theorem}
\newtheorem{lema}[teo]{Lemma}
\newtheorem{cora}[teo]{Corollary}
\newtheorem{propa}[teo]{Proposition}

\theoremstyle{drem}
\newtheorem{rema}[teo]{Remark}

\theoremstyle{ddef}
\newtheorem{dfa}[teo]{Definition}

\def\eqtag{\addtocounter{teo}{1} \tag{\theteo}}

\def\vs{\vrule width 0cm height 0.1in depth 0in}

\def\fr#1#2{ \frac{\displaystyle \vs #1}{\displaystyle \vs #2}}

\def\bint#1#2{ {\displaystyle\vs \int_{#1}^{#2}} }

\def\ie{\emph{i.e. }}
\def\eg{\emph{e.g. }}
\def\cf{\emph{cf. }}

\def\srl#1{\overline{#1}}

\def\eps{\epsilon}
\def\bv{\! \big|}

\def\vide{\varnothing}
\renewcommand{\setminus}{\smallsetminus}

\def\supp#1{\textrm{\raisebox{.5ex}{\mbox{$\underset{#1}{\sup}$}}} \:}

\def\dd{\mathrm{d} \!}
\def\del{\partial \!}
\def\db{\bar{\partial} \!}

\def\zb{\srl{z}}

\def\nn{\mathbb{N}}

\def\rr{\mathbb{R}}
\def\cc{\mathbb{C}}

\def\cp{\mathbb{C}\mathrm{P}}

\def\Id{\mathrm{Id}}

\def\ker{\mathrm{Ker}\,}

\def\img{\mathrm{Im}\,}
\def\endo{\mathrm{End}}
\def\homo{\mathrm{Hom}}

\def\diam{\mathrm{Diam}\,}
\def\injrad{\mathrm{injrad}\,}

\def\sm22#1#2#3#4{\left( \begin{smallmatrix}   #1 & #2 \\   #3 & #4 \\  \end{smallmatrix} \right)}

\def\un{1 \!\! \mathrm{l}} 

\def\ssi{\Leftrightarrow}
\def\imp{\Rightarrow}
\def\inj{\hookrightarrow}

\def\id#1{\mathfrak{#1}}
\def\jo#1{\mathcal{#1}}
\def\scr#1{\mathscr{#1}}

\def\nr#1{\left\| #1 \right\|}
\def\pnr#1{\| #1 \|}

\def\abs#1{\left\lvert #1 \right\rvert}
\def\babs#1{\big\lvert #1 \big\rvert}

\def\gen#1{\left\langle #1 \right\rangle}
\def\pgen#1{\langle #1 \rangle}

\def\lun{ { \mathsf{L}^{\!\!^1} } }

\def\ldeu{ { \mathsf{L}^{\!\!^2} }}
\def\lpe{ { \mathsf{L}^{\!\!^p} } }
\def\linf{ {\mathsf{L}\!\!^{\!^\infty}} } 
\def\czer{ {\mathsf{C}^0} }
\def\cinf{ {\mathsf{C}^\infty} }

\def\wud{ {\mathsf{W}^{\!^{1,2}}} }

\def\somme#1#2{\overset{#2}{\underset{#1}{\sum}}}

\def\union#1#2{\overset{#2}{\underset{#1}{\cup}}}
\def\inter#1#2{\overset{#2}{\underset{#1}{\cap}}}

\def\tg{\mathrm{T}}
\def\exp{\mathrm{exp}}

\def\habr{H^r_{a,b}}

\def\SR{\Sigma}

\def\rmxi{r_{i,\textrm{max}}}
\def\rmin{r_\textrm{min}}

\newcounter{rind}
\newcounter{cind}

\title{A Runge approximation theorem for pseudo-holomorphic maps}
\author{Antoine Gournay}
\date{~}

\begin{document}
\selectlanguage{english}

\maketitle

\begin{abstract}
The Runge approximation theorem for holomorphic maps is a fundamental result in complex analysis, and, consequently, many works have been devoted to extend it to ohter spaces (\eg maps between certain algebraic varieties or complex manifolds). This article presents such a result for pseudo-holomorphic maps from a compact Riemann surface to a compact almost-complex manifold $M$, given that the manifold $M$ admits many pseudo-holomorphic maps from $\cp^1$ which can be thought of as local approximations of the Laurent expansion $az + b r^2/z$. These result specialize to some compact algebraic varieties (\eg rationally connected projective varieties). An application to Lefschetz fibrations is presented.
\end{abstract}

\section{Introduction}
\indent The Runge approximation theorem for holomorphic maps ($U \to \cc$) is a fundamental result in complex analysis. The aim of this article is to prove such a result for (pseudo-)holomorphic maps from a compact Riemann surface to a compact (almost-)complex manifold $M$ under certain assumptions. Though the setting is definitively that of pseudo-holomorphic maps, it also covers some complex varieties.

\subsection{Problem, assumption and result}

\indent {\bf Basic concepts.} A manifold $M$ of even real dimension is said to be almost complex when it is endowed with a section $J \in \endo \tg M$ such that $\forall x \in M, \, J_x^2 = -\Id_{\tg_x M}$. Complex multiplication gives rise to such a structure, and when $M$ is of real dimension $2$ an almost complex structure is a complex structure (as can be seen from the vanishing of the Nijenhuis tensor). Throughout the text, $M$ will be compact and $\SR$ will denote a compact Riemann surface whose complex structure will be written $j$.
\par A map $u:\SR \to (M,J)$ will be said pseudo-holomorphic or $J$-holomorphic if $\dd u \circ j = J \circ \dd u$, or, equivalently, if 
\[
\forall v \in \tg_z \SR, \qquad \db_J u (v) := \frac{1}{2} ( \dd u_z (v) + J_{u(z)} \circ \dd u_z \circ j_z(v) ) =0.
\]
\par {\bf Problem.} The Runge approximation problem can, in this setting, be formulated as follows: given a $J$-holomorphic map $f:U \to (M,J)$ for $U$ an open subset of $\SR$, a compact $K \subset U$, some small $\delta \in \rr_{>0}$, under which conditions is it possible to find a $J$-holomorphic map $h: \SR \to (M,J)$ such that $\|h-f\|_{\czer(K)} <\delta$?
\par Though the interest of the problem lies in the fact that $h$ is defined on the whole of $\SR$, this in not actually so much an extension result (which is in general impossible even for holomorphic maps $\cc \to \cc$) as an approximation result (whence the name). But even then, there are choices of $(M,J)$ and $\SR$ where it is impossible (see below). The subject matter of this article is to show that under certain assumptions on $(M,J)$, the aforementioned question has a positive answer for any $\SR$. 

\par {\bf Assumption.} The basic tool that is required by the present method concerns local expansion. To say things simply, assume $M$ is complex. Then the working hypothesis, that will henceforth be referred to as the \emph{double tangent property}, is that at (almost) every point $m \in M$ and for (almost) any pair of tangents $(a,b)$ there must be a holomorphic map $\cp^1 \to M$ with local (Laurent) expansion $az + b r^2/z + O(r^{1+\eps})$ in some annulus. For a precise statement, see definition \ref{ddtp}.

Furthermore, the almost complex structure has to be assumed \emph{regular} (as described in McDuff and Salamon's book \cite[Theorem 3.1.5]{mds1}). Regularity is important  to ensure that the linearization of the $\db$ operator at a pseudo-holomorphic curve ($\cp^1 \to (M,J)$) is surjective, thence invertible. If this is not assumed, then each grafting might generate additional problem. From an algebraic viewpoint, this implies that fusion of rational curves (the construction which to two curves $x=0$ and $y=0$ associates the curve $xy = \eps$) is possible.

\begin{teoa}\label{runge}
Let $(M,J)$ be an almost complex manifold that has the double tangent property, and assume $J$ is regular. Then for all $U \subset \SR$ open, all $J$-holomorphic map $f:U \to (M,J)$, all $K \subset U$ compact and all $\delta >0$, there is a $J$-holomorphic map $h: \SR \to (M,J)$ such that $\|h-f\|_{\czer(K)} < \delta$ provided there is a $\czer$ extension of $f$ to $\Sigma$. 
\end{teoa}

\par Though apparently very constraining, proposition \ref{minim} indicates that the assumptions are rather minimal. 

\par {\bf Related works.} Runge approximation has already been the source of interest for maps between other objects. Before listing some of these work, it should be noted that the source is, contrary to the present paper, a non-compact space (compact affine algebraic varieties and compact Stein manifolds are union of points). Demailly, Lempert and Shiffman \cite[Theorem 1.1]{DLS} and Lempert \cite[Theorem 1.1]{Lem} (a proof of an algebraic nature of the latter is presented in Bilski's article \cite{Bil}) obtain stronger Runge approximations: for a map $f$ defined $K$ on a holomorphically convex compact in an affine algebraic variety with values in a quasi-projective variety, the approximating map are algebraic Nash maps (a stronger condition than simply holomorphic). The condition that $K$ is holomorphically convex is necessary as the source might be of higher dimension. Kucharz \cite[Theorem 1]{Kuc} also gives such approximations when the target space is a Grassmannian (the source being again an affine algebraic variety); depending on the conditions satisfied by the initial map, the approximation is algebraic or regular. There is also the Oka-Weil approximation theorem: it states that Runge approximation holds for functions on holomorphically convex compacts of Stein manifolds with values in $\cc$. For more results in this direction (\eg when the target is an Oka manifold), it seems wisest to refer the reader to a recent survey by Forstneri\v{c} and L\'arusson \cite{FL}. There is however a tempting analogy to make: in an Oka manifold there are lots of maps from $\cc$, and these maps allow Runge approximation when the target is an Oka manifold, (the source is Stein, hence non-compact) and there are no obvious topological obstructions. In the case at hand, the target manifold is required to admit lots of maps from $\cp^1$ and no topological obstructions, in order to admit Runge approximation from a compact Riemann surface. Finally, Runge approximations have been studied for operators which are not the usual $\db$ operator (\ie holomorphic functions), for example, by Brackx and Delanghe \cite{BD} (the source here is Euclidean space and the target a Clifford algebra). 

\subsection{Examples and applications.} 

\par {\bf Examples.} A simple example in which the hypothesis in theorem \ref{runge} are easily verified is $M = \cp^n$ with its usual complex structure (note that the classical Runge theorem may, of course, directly be applied in this case). On the other hand, $M = \mathbb{T}^n$ with their usual complex structures are clearly cases where it fails, as there can be no holomorphic maps from $\cp^1 \to \mathbb{T}^n$. In this particular example, this is not only that the hypothesis of theorem \ref{runge} cannot be fulfilled. The Runge approximation in $\mathbb{T}^n$ cannot exist for  $\SR = \cp^1$; it could however still be true for other Riemann surfaces $\SR$, \eg $\SR = \mathbb{T}^1$. 

\par The condition of the double tangent property, as expressed in terms of Laurent expansion, is a bit awkward. Fortunately, it is implied by more tractable properties.

\par A map is said to \emph{realize the tangent} $v \in \tg M$ if $v$ is in the image of the differential, or, as expressed in local charts, if it can be written as $v z + O(|z|^2)$ (see Sikorav's characterization of local behavior in \cite[Proposition 3]{Sik}). Obviously, if there is a map realizing $v$ then, $\forall \lambda \in \rr$, there is a map realizing $\lambda v$. Denote by $\mathrm{S}M$ the unit tangent bundle of $M$. The following proposition is a direct consequence of \cite[Theorem 1.3 and \S{}2]{moi-cyl}.

\begin{propa}\label{tdtp}
Let $(M,J)$ be an almost complex manifold such that $J$ is of class $\mathsf{C}^2$ and regular, and there is a dense set $D$ in $\mathrm{S}M$ such that $\forall v \in D$, $v$ is realized by a pseudo-holomorphic map $\cp^1 \to M$. Then $M$ has the double tangent property.
\end{propa}

\par Though gluing two pseudo-holomorphic curves is possible for any regular $J$, the $\mathsf{C}^2$ condition on $J$ is required in \cite{moi-cyl} to obtain the local Laurent expansion (it can be weakened to $\mathsf{C}^{1,1}$, \ie $\mathsf{C}^1$ and of Lipschitz derivative). 

\par The conditions of proposition \ref{tdtp} (and, consequently, of theorem \ref{runge}) also hold in a Grassmanian $\scr{G}(k,E)$. Indeed $\tg_A \scr{G} \simeq \homo(A,B)$ for $B$ a supplement of $A = [a_1 \wedge \ldots \wedge a_k] \in \scr{G}$. For $p \in \homo(A,B)$, the map $u : z \mapsto [\big(a_1+z p(a_1)\big) \wedge \ldots \wedge \big(a_k+z p(a_k)\big)]$ extends to $\cp^1$ and realize the tangent $p$. As compact affine algebraic varieties are union of points, this complements (though the approximating functions are only holomorphic) the result of Kucharz \cite{Kuc}.

\par Similarly, if a ``non-standard'' complex structure on $\cp^n$ is taken, and that this complex structure remains tamed by the symplectic form, a result of Gromov \cite{gro82} implies that proposition \ref{tdtp} and theorem \ref{runge} hold.

\par For a more general approach to complex varieties $M$ that will satisfy the assumptions, the reader is referred to Debarre's book \cite[Chapter 4]{Deb} among numerous references concerning rationally connected varieties. As such, if there is a ``free'' curve (see \cite[Definition 4.5]{Deb}) and $M$ is a smooth quasi-projective variety, then the evaluation map (from the moduli space of curves $\cp^1 \to M$; see \cite[Proposition 4.8]{Deb}) is smooth and its image is dense. An intuitive description is that the ``free'' curve can be deformed so as to pass through almost any point of $M$. Since the regularity of $J$ amounts to the surjectivity of the differential of the evaluation map, such varieties are natural candidates fir the application of the theorem. If the hypothesis are further strengthened to the existence of a ``very free'' curve (see \cite[Definition 4.5]{Deb}) in $M$ and that $M$ is a smooth projective variety, then $M$ is rationally connected (see \cite[Definition 4.3 and Corollary 4.17]{Deb}). Over $\cc$, rationally connected varieties are exactly those for which a general pair of points (outside a subvariety of codimension at least $2$) can be joined by a rational curve (see \cite[Remark 4.4.(3)]{Deb}). Consequently, for $M$ a rationally connected smooth quasi-projective variety with a free curve (or, in particular, a smooth projective variety with a very free curve), proposition \ref{tdtp} and theorem \ref{runge} hold.

\par Compactified moduli spaces of curves of genus $g$ (we speak of the Deligne-Mumford compactification), $\overline{\jo{M}}_g$, are unirational when $g \leq 14$, and rationally connected for $g \leq 15$. As a consequence theorem \ref{runge} will apply for these spaces. However, if $g \geq 24$, the moduli space is then of general type (see the survey of Farkas \cite{Far} on the topic; some further results are present in the paper of Ballico, Casnati and Fontanari \cite{BCF}). 

\par {\bf Application.} A case of interest for the application of theorem \ref{runge} are Lefschetz fibrations; this idea is due to S.~Donaldson. The aim is to partially recover  the results of Auroux  (see \cite{Aur} and \cite{Aur2}) and Siebert-Tian \cite{ST}. A fibration $p:V \to \cp^1$ can be seen in terms of its classifying map $\cp^1 \to \overline{\jo{M}}_2 $ where $\overline{\jo{M}}_2$, the (Deligne-Mumford compactification of the) moduli space of genus $2$ curves, is (almost-)smooth and complex (actually Kählerian). For more details on this construction, the reader is referred to I.~Smith's paper \cite{Smi}. In this context, the Runge theorem \ref{runge} applies: as mentioned above $\overline{\jo{M}}_2$ satisfies the hypothesis. Taking $U= \emptyset$ and reinterpreting the method of the proof in this context (the statement of the theorem alone does not imply the upcoming statement), one gets that any Lefschetz fibration becomes, after sufficiently many fibred sum (stabilization), holomorphic. Thence

\begin{cora}
Let $p:M \to \cp^1$ be a genus $g \leq 15$ differentiable Lefschetz fibration. Then, after fiber sum with sufficiently many copies of some holomorphic Lefschetz fibrations (\emph{a.k.a.} stabilization), it becomes isomorphic to a holomorphic Lefschetz fibration.
\end{cora}

\par A comment that the authors owes to I.~Smith is that this is perhaps even more striking in view of \cite{Aur2}. Indeed, Auroux's method do not require any hypothesis on the genus of the surface; the methods are in fact much more direct (the ``universal'' fibration $f^0_g$ is quite explicit). This could hint at many things: that there might be a dense set of tangents realized by rational curves in $\srl{\jo{M}}_g$, while this space remains of generic type, or that it could be possible to restrict the problem on a part of $\srl{\jo{M}}_g$ having this property.

\par In the classical Runge theorem, the number of poles of the approximating map is related to the topology of the set $U$. Unfortunately, the notion of a pole does not have a meaning in the compact setting. What will obviously happen however is that one expects that the energy (the $\ldeu$ norm of the differential) of the approximating map may be very big. A consequence of Taubes result \cite[Theorem 1.1]{tau} is that the minimal number of necessary connect sums of $\overline{\cp^2}$ required to make a metric structure anti-self-dual is defined. It is an invariant of the conformal metric, but not a simple one to compute (LeBrun and Singer \cite[\S{}1]{LS} gave a bound of $14$ in the case of $\cp^2$ with its usual metric). Though again probably not an easy question to answer, it would, in the context of the present article, be interesting to look for the minimal energy of a $J$-holomorphic map realizing a given approximation.

\par In this perspective, there is another interesting consequence concerning surfaces $X$ that are smooth fiber bundles over a base $B$ a curve of genus $\geq 2$ and whose fibers are curves of genus $g$. Then as long as $g \leq 15$ (so that $\srl{M}_g$ is rationally connected) the classifying map can be approximated (stabilized) into a holomorphic one and the corresponding surface $X'$ possesses a complex structure. On the other hand, if $g \geq 2$ the hypothesis of a theorem of Kotschick \cite[Theorem 3]{Kot} hold. Consequently, for these genera ($2 \leq g \leq 15$), if $\sigma$ is the signature and $\chi$ the Euler characteristic, $3|\sigma(X')|\leq |\chi(X')| = |\chi(X)|$. This could, the signature being additive, give a lower bound on the minimum number of surgeries required. 

\par {\bf On the hypothesis.} Before getting to the heart of the matter, it is worth noting that the hypothesis of theorem \ref{runge} are (keeping proposition \ref{tdtp} in mind) as minimal as can be reasonably expected.

\begin{propa}\label{minim}
Let $D_1 \subset \cp^1$ be an open disk and $\srl{D}_{1/2}$ a smaller closed disk. Suppose that $J$ is Lipschitz. Suppose that for every $\delta >0 $ and every $J$-holomorphic map $f: D_1 \to (M,J)$ there exist a $J$-holomorphic map $h:\cp^1 \to (M,J)$ such that $\|f-h\|_{\czer(\srl{D}_{1/2})} < \delta$. Then there exists a dense subset $R \subset \mathrm{S}M$, so that $\forall v \in R$ there is a $J$-holomorphic map $g_v : \cp^1 \to M$ realizing the tangent $v$.
\end{propa}
\begin{proof}
Since there exists a map $f:D_1 \to (M,J)$ realizing any tangent $v \in \tg_x M$ (there is no local obstruction to pseudo-holomorphic maps, see Sikorav's presentation \cite[Theorem 3.1.1.(i)]{Sik2}), there must be a map $h$ approximating it on $\srl{D}_{1/2}$. Using some dilation the discs can be assumed small and since locally things are close to the holomorphic context ($J$ is Lipschitz), the Cauchy integral formula will give a $\mathsf{C}^1$ approximation from the $\czer$ one. Consequently, $h$ (up to a reparametrization to get a unit vector) can be made to approximate the tangent $v$.
\end{proof}

\subsection{Sketch of the proof}  
The core of the problem is to solve the non-linear equation $\db_J u =0$. This will be done by constructing an approximate solution and then developing an implicit function theorem to deform the approximate solution in a true solution. The methods follow those of Taubes \cite{tau}.

\par To sketch the path employed here, it is good to think of the (modified) Newton's method is employed to solve a non-linear equation $h(x)=0$. Suppose for simplicity that $h(0)$ is almost a solution, then Newton's (modified) iteration is 
\[
N(x) = x - h'(0)^{-1} h(x) = -h'(0)^{-1} \Big(h(0) + \big(\underset{q(x)}{\underbrace{h(x) - x h'(0)-h(0)}}\big) \Big).
\]
The term $q(x)$ represents the higher order variations of the function. The iteration will work if
\begin{itemize}
\item $N$ is contracting, \ie 
\[ \eqtag \label{newt}
\begin{array}{rcl}
\pnr{N(x) - N(x')} & \leq & \eps \pnr{x-x'} \\
\ssi \pnr{h'(0)^{-1} \big( q(x) - q(x') \big)} & \leq & \eps \pnr{x-x'}.
\end{array}
\]
 for $x$ and $x'$ in a ball $B_r$.
\item $0$ is an almost solution of $h$, \ie $\pnr{h(0)} \leq r(1-\eps)$.
\end{itemize} 
There will then be convergence (for the norm) of the sequence $x_{n+1} = N(x_n)$ (and $x_0=0$) to a fixed point $N(x) = x$ which is a solution of $h(x) = 0$. Though it may be naive, this presentation has the advantage of summing up all the key ingredients. To solve $\db_J u =0$ an approximate solution must be constructed, an inverse to the linearization realized and proper norms chosen.

\par {\bf The approximate solution.} The heuristic idea to construct the approximate solution can be found in Donaldson's paper \cite[\S{}3]{don}; it is described here in {\S}\ref{greffe}. Given a $J$-holomorphic map $g_0:U \to M$ from a complex open set to an almost-complex manifold $(M,J)$, it is always possible (given there is a $\czer$ extension) to extend it by a $\cinf$ map $g:\SR \to M $ defined on the whole Riemann surface and identical to the former when restricted to a compact subset of $U$. There is a set, presumably quite large, where this map is not $J$-holomorphic. In order to get a holomorphic map from this one, the idea is to change the definition of the function on small discs. On these discs one would like to replace it by a $J$-holomorphic map having a behavior on the boundary of the disc close to that of $g$ the rough $\cinf$ extension of $g_0$. 

\par In an almost-complex manifold $(M,J)$, the idea is to proceed as follows. Let us be at a point where $\db_J g \neq 0$, and let us consider local charts at the source and the image so that the almost complex structure induces the endomorphism $i$ on $\cc^n$. The rough extension $g$ can be written as $g(z)=az+b\zb + O(\abs{z}^2)$. It is of course impossible to approximate this by a holomorphic map. However, suppose there is a $J$-holomorphic function $h$ such that $h(z)=az+b\frac{r^2}{z} + o(|z|^2)$ around $|z|=r$ (as mentioned above the hypothesis ensures that such a map exist). This is a possible approximation of $az + b \zb $ when $|z| \simeq r$. The strategy is to graft $h$ to $g$ along this circle, and to repeat this operation until the set of points where $g$ is not $J$-holomorphic is small.

\par {\bf Inverting the linearization.} The linearization of the operator $\db_J$ around a map $u: \SR \to (M,J)$ is described in McDuff and Salamon's book by equation \cite[(3.1.4)]{mds1}. It is a linear map $D_u$ sending sections $\xi$ of the bundle $u^*\tg M$ to $1$-forms on the same bundle:
\[
\textrm{for } v \in \tg_z \SR, \qquad D_u \xi (v) = \frac{1}{2} ( \nabla_v \xi(z) + J_{u(z)} \nabla_{j(v)} \xi(z) ) + \frac{1}{2} J_{u(z)} \big( \nabla_v J_{u(z)} \big) \del_J u (v),  
\]
where $\del_J u :=\frac{1}{2}( \dd u - J \circ \dd u \circ j)$. It is noteworthy that the differential of the function $u$ enters in this expression. Indeed this will force us to take more care in the construction of the approximate solution: the differential will have to remain bounded. Another important property of this linearization is the highest degree term in $D_u D_u^*$ which is the Laplacian. (When $M$ is K{\"a}hlerian, there is actually a Weitzenb{\"o}ck formula.)

\par The inversion of this linearization will be done first by decomposing the problem in different parts (in \S{}\ref{Ts6}): the analysis will be conducted separately on each disk where the rough extension $g_0$ has been modified and on the original $\SR$. On the disks things will go relatively without much problems, but on $\SR$ it will be necessary to solve only up to ``small'' eigenvalues of the Laplacian (see \S{}\ref{Ts7}). 

\par This failure take into account the small eigenvalues will prolong the proof further, but will be deferred after the argument that can properly be interpreted as Newton's iteration. Indeed, instead of constructing one approximate solution, a family of them (parametrized by a certain subset of the space of small eigenvalues) will then be considered. Interpreted as a composition of maps ``small eigenvalues'' $\to$ ``approximate solutions'' $\to$ ``small eigenvalues'', the presence of a fixed point will allow to conclude that there is an actual solution (see \S{}\ref{Ts9}).   

\par {\bf Norms.} An approach using Sobolev or H{\"o}lderian norms seems to be bound to fail in this situation. Here are two reasons. First, the H{\"o}lderian norm contains a $\czer$ component, and our approximate solution, is not an approximate solution in the $\czer$ sense. Second, the $\lpe$ norm of the differential of the approximate solution, $\dd g$, will not be bounded. Indeed on each disc where a surgery occurs, this norm increases by a quantity which is {\it a priori} significative and the number of these surgeries is not bounded. This seems to indicate that other norms are required; norms which depend on a $\sup$ rather than an integral over the whole surface, but that also do some averaging so that being bounded on a small region gives a small norm. The norms of Taubes are also convenient because the inversion of the linearized operator is done through the Laplacian. When the norms behave ``well'' with respect to the inverse of the linearization, one expects \eqref{newt} to give more easily the desired estimate. Suppose that $h'(0)^{-1}$ is bounded for the norms in said equation, then the estimate boils down to $\pnr{q(x)-q(x')}$, which (again only morally) could be expected to have an upper bound in $(\pnr{x}+\pnr{x'}) \pnr{x-x'}$. This (in \S{}\ref{Ts8}) for $\pnr{x}$ and $\pnr{x'}$ sufficiently small yields the contraction.  
 
\par Section \ref{chanell} establishes the properties that will be required in order to work with these norms. In the present article we will however not dwell on the regularity of the solutions, these questions, which are quite standard, have already been addressed by Donaldson in \cite[\S{}2.4]{don} and also by Matsuo and Tsukamoto in \cite[\S{}4.2]{MT}. These norms are also quite reminiscent of the Kato class condition, see Simon's survey \cite[p.3528, paragraph (e)]{Sim}; the interested reader can also find references as to why a choice of convolution norms might be appropriate.

\par {\bf Instantons and anti-self-dual metrics.} As a last note in this introduction, there are some differences between the case of instantons (on the sphere) or the anti-self-dual metrics and the $J$-holomorphic problem: the non-linearity is quadratic in instantons, whereas it does not seem to have any particular behavior in the latter. Furthermore, whereas  gluing in instantons does not affect the equation to be solved, grafting pseudo-holomorphic curves has an effect both on the linear and non-linear terms. The scenario is thus closer to that of anti-self-dual metrics in dimension $4$, studied in Taubes' paper \cite{tau}; it is nonetheless easier as the equation to be dealt with is of the first order rather than of order 2 and the symmetry group is finite dimensional rather than infinite dimensional. Furthermore, in our case, the linearization is a linear elliptic operator. But Taubes' norms prove to be useful through their clever use of the Laplacian; and in \cite[\S{}5]{tau}, even if the linearization is not elliptic, the method still applies. 

{\it Acknowledgments:} M.~Le~Barbier and P.~Pansu are warmly thanked for their questions, comments and suggestions. 

\section{Elliptic analysis à la Taubes} \label{chanell}

\par This section contains an adaptation of Taubes ``toolbox'' \cite[\S{}4]{tau} in dimension $2$. Taubes' norm do not behave as nicely in dimension $2$ as in higher dimensions: Green's kernel has a logarithmic singularity, the bound obtained in theorem \ref{exborn} contains a logarithm which in higher dimension is but a constant. It will however not be of much consequence. Indeed, in the inversion of the linear operator (see \S{}\ref{Ts7}) much more daunting terms will appear.
\par A worthwhile suggestion of Taubes (that will not be explored further here), in dimension $2$, is to use $\wud$ norms together with a norm of Morrey type (see for example one of Taylor's book \cite[\S{}A.2]{Tay}) \ie
\[
\nr{f}_{\scr{M},\rho} = \supp{x \in \SR} \supp{r \in [0,\rho]} \Big( \frac{\rho^2}{r^2} \int_{B_r(x)} |f(y)|^2 \dd y \Big)^{1/2}.
\] 
Indeed, in low dimensions, convolution are not necessarily most appropriate (see \cite[paragraph (d), p.3528]{Sim} in Simon's survey and the reference to Sturm therein).

\subsection{Definitions and properties of the norms} \label{Ts4}
\indent As said above Sobolev norms are unfortunately not appropriate for our problem. Still it is important to have norms which take into account the point-wise behavior of maps. The norms introduced here look like an $\linf$ norm but applied to the inverse of the Laplacian (the convolution with Green's kernel). \\
\begin{dfa} \label{normes1}
Let $\rho \in ]0,e^{-1}[$. Let $x \in \SR$, $B_\rho(x)$ be the open ball of radius $\rho$ centered at $x$. Define 
\[
\begin{array}{lcl}
\nr{u}_{\linf}         &=& \supp{x \in \SR} \abs{u(x)}, \\
\nr{u}_{*,\rho}       &=& \supp{x \in \SR} \bint{B_\rho(x)}{} \ln(d(x,y)^{-1}) \abs{u(y)} \dd y \\
\nr{u}_{2*,\rho}      &=& \supp{x \in \SR} \left[ \bint{B_\rho(x)}{} \ln(d(x,y)^{-1}) \abs{u(y)}^2 \dd y \right]^{1/2}\\
\nr{u}_{\jo{L}^0,\rho}&=& \nr{u}_{\linf}  + \nr{\nabla u}_{2*,\rho}
\end{array}
\]
\end{dfa}
\indent These norms will not be sufficient for our needs, a seminorm $\jo{L}^1$ will arise naturally; it can be seen as an ``integration by parts'' norm: although derivatives do not appear explicitly, they are nevertheless measured in it. A parenthesis is necessary for their introduction. 
\par Denote by  $\jo{S}(\tg^*\SR \otimes V) \subset \cinf(\tg^*\SR \otimes V)$ the subset of elements of $\jo{L}^0$ norm equal to $1$. Furthermore, given local charts around $x$, then for $\rho$ sufficiently small, $B_\rho(x)$ identifies to an usual ball of $\rr^2$. In these coordinates, a section of $\tg^*\rr^2 $ can be written as a map $\rr^2 \to \rr^2$. Next, notice that maps from the circle $\rr^2 \supset S^1 \to \rr^2$ extend to maps independent of the radial coordinate $\rr^2 \setminus \{0\} \to \rr^2$. Last, denote by $\Gamma=\{f \in \cinf(S^1,\rr^2) | \nr{f}_{\ldeu} = 1 \}$. 
\begin{dfa} \label{normes2}
Let $\rho\in ]0,e^{-1}[$ be less than the injectivity radius, the seminorm $\jo{L}^1 $ associated to $u\in \cinf(V)$ 
\[
\nr{u}_{\jo{L}^1,\rho} = \supp{x\in \SR} \supp{v \in \jo{S}(\tg^*\SR \otimes V)} \supp{\phi \in \Gamma} \bint{B_\rho(x)}{} \fr{\gen{v,\phi \otimes u}(y)}{d(x,y)} \dd y
\]
and enters in the definition of the following two norms:
\[
\begin{array}{lcl}
  \nr{u}_{\jo{L},\rho} &=& \nr{u}_{\jo{L}^0,\rho} + \nr{\nabla u}_{\jo{L}^1,\rho}, \\
  \nr{u}_{\jo{H},\rho} &=& \nr{u}_{2*,\rho} + \nr{u}_{\jo{L}^1,\rho}.
\end{array}
\]
\end{dfa}
Here are some elementary properties of these norms.
\begin{prop} \label{l0ssmul} Suppose that $\rho \in ]0,e^{-1}[$.
  \begin{enumerate} \renewcommand{\labelenumi}{{\normalfont \alph{enumi}.}}
  \item $\nr{ab}_{*,\rho} \leq \nr{a}_{2*,\rho}\nr{b}_{2*,\rho}$
  \item $\nr{\cdot}_{\lun(B_\rho(x)) } \leq |\ln \rho|^{-1} \nr{\cdot}_{*,\rho}$ and $\nr{\cdot}_{\ldeu(B_\rho(x)) } \leq |\ln \rho|^{-1/2} \nr{\cdot}_{2*,\rho}$.
  \item If $k \in \rr_{>1}$, $k\rho < e^{-1}$ then 
\[
\begin{array}{lclcl}
 \nr{\cdot}_{*,\rho} & \leq & \nr{\cdot}_{*,k\rho} & \leq & 4k^2 \nr{\cdot}_{*,\rho} ,\\ 
 \nr{\cdot}_{2*,\rho} & \leq & \nr{\cdot}_{2*,k\rho} & \leq & 2k \nr{\cdot}_{2*,\rho},\\
\nr{\cdot}_{\jo{L}^1,\rho} & \leq & \nr{\cdot}_{\jo{L}^1,k\rho} & \leq & 4k \nr{\cdot}_{\jo{L}^1,\rho}. \\
 \end{array}
\]
  \item The norm $\jo{L}^0$ is sub-multiplicative: $\nr{v \otimes w}_{\jo{L}^0,\rho} \leq \nr{v}_{\jo{L}^0,\rho}\nr{w}_{\jo{L}^0,\rho}$.
  \end{enumerate}
 \end{prop}
\begin{proof} 
The first of these properties is a direct consequence of Hölder's inequality.
\par Whereas the second one follows from
\[
\nr{u}_{\lun(B_\rho(x)) } \leq |\ln \rho|^{-1} \bint{B_\rho(x)}{} |u(y)| |\ln \rho|\dd y \leq |\ln \rho|^{-1} \bint{B_\rho(x)}{} |u(y)| |\ln d(x,y)| \dd y;
\]
the $2*$ case being identical.

\par As for the third, the norms $*, 2*$ and $\jo{L}^1$ are obtained by the $\sup$ of integrals on balls, the ratio of areas allows us to bound the integral taken on  a large ball by those computed on smaller balls. The square root of area ratio works for $2*$ and for $\jo{L}^1$ one can get a better bound as the weight $1/|\cdot|$ is rapidly decreasing. 

\par The last property is again a simple calculation: 
\[ \qquad \quad
\begin{array}{rllr}
\nr{v \otimes w}_{\jo{L}^0,\rho} 
          &\leq& \nr{v}_{\linf} \nr{w}_{\linf} + \nr{\nabla v \otimes w + v \otimes \nabla w }_{2*,\rho} \\
          &\leq& \nr{v}_{\linf} \nr{w}_{\linf} + \nr{\nabla v \otimes w}_{2*,\rho} + \nr{v \otimes \nabla w }_{2*,\rho} \\
          &\leq& \nr{v}_{\linf} \nr{w}_{\linf} + \nr{\nabla v }_{2*,\rho} \nr{w}_{\linf} +\nr{v}_{\linf} \nr{\nabla w }_{2*,\rho} \\
          &\leq& \nr{v}_{\jo{L}^0,\rho}\nr{w}_{\jo{L}^0,\rho} & \qquad \quad \qedhere
\end{array}
\]
\end{proof}
\indent Before giving estimates with these norms, the following lemma, describing the difference between Green's kernel for the Laplacian (with a singularity at $x$) and the function $\ln d(x,\cdot)^{-1}$, must be established.
\begin{lema} \label{difgreen}
Given $x \in \SR$, let $G(x,\cdot): \cinf(\SR \setminus \{x\})$ be Green's function for $\nabla^* \nabla +1: \cinf(\SR) \to \cinf(\SR)$. $\exists c_2 \in \rr_{>0}$ depending on the diameter of $\SR$ such that
\[
\begin{array}{rcl}
  \abs{G(x,\cdot) + (2\pi)^{-1} \ln(d(x,\cdot))} &\leq & c_2 \abs{d(x,\cdot)^2 \ln d(x,\cdot)} \\
  \abs{\nabla G(x,\cdot) + (2\pi d(x,\cdot))^{-1} \nabla d(x,\cdot) } &\leq & c_2 \abs{d(x,\cdot) \ln d(x,\cdot)} \\
  \abs{\nabla^* \nabla G(x,\cdot)} &\leq & c_2 d(x,\cdot)^{-2}.
\end{array}
\]
\end{lema}
\begin{proof}
For this proof, it is recommended to (re)read the important results on Green's function; see for example Aubin's book \cite[ch4 §2.1-§2.3]{Aub}. It is well-known, but presented here as the case $n=2$ is often omitted. Start by writing the Laplacian for a function depending only on polar (geodesic) coordinates (\cf \cite[4.9]{Aub}):
\[
\Delta \phi(r) = \phi'' + \frac{1}{r} \phi' + \phi' \del_r \ln \sqrt{\abs{g}},
\]
where $g$ is the metric; an useful bound of the term where it plays a role is $ \del_r \ln \sqrt{\abs{g}} \leq K_1 r $ for $K_1 \in \rr_{>0}$, see \cite[Theorem 1.53]{Aub}. Let $f(r):\rr_{\geq 0} \to [0,1]$ be a smooth function which is $0$ if $r > \injrad \SR$ and equal to $1$ when $r<\injrad \SR /2$. Furthermore, take $r = d(x,y)$ and define the parametrix
\[
H(x,y) = -(2 \pi)^{-1} f(r) \ln r.
\]
A direct calculation shows that
\[
\Delta_y H(x,y) = f''(r) \ln r + f'(r) r^{-1} (2+ \ln r) + (f'(r) \ln r + f(r) r^{-1}) \del_r \ln \sqrt{\abs{g}}.
\]
Thanks to the bound on the last term and since $f'(r)=f''(r)=0 $ when $r< \injrad \SR /2$, there exists a constant $K_2$ (depending on the injectivity radius and the choice of $f$) such that
\[
\abs{\Delta_y H(x,y)} \leq K_2.
\]
This said, the first inequality follows from equation \cite[(4.17)]{Aub}; let $\Gamma_1(x,y) = - \Delta_y H(x,y)$, let $\Gamma_{i+1} = \int_\SR \dd vol(z) \Gamma_i (x,z) \Gamma_1(z,y)$ and let $F_k(x,y)$ be defined by $\Delta_y F(x,y) = \Gamma_k(x,y) - (\int_\SR \dd vol)^{-1}$. With these notations, $\forall k \in \nn_{\geq 2}$,
\[
G(x,y) = H(x,y) + \somme{i=1}{k} \bint{\SR}{} \dd vol(z) \Gamma_i(x,z) H(z,y) + F_{k+1}(x,y).
\]
The term $i=1$ will have the most singular behavior at $0$. However, since $\Gamma_1(x,y) = -\Delta_y H(x,y)$ is bounded and since $H(x,y)$ is essentially a logarithm of the distance, a positive real number $K_3$ which depends on the diameter exists so that
\[
\abs{\bint{\SR}{} \dd vol(z) \Gamma_1(x,z) H(z,y) } \leq K_3 r^2 \abs{\ln r}.
\]
The estimations of the derivatives are obtained likewise.
\end{proof}

\subsection{Estimation on the solutions of $\delta^* \delta u = \chi$.}
Let $V$ and $W$ be vector bundles on $\SR$ having the same dimension. Let $\delta : \cinf(V) \to \cinf(W)$ be an elliptical operator of order $1$. Let $\sigma \in \homo(\tg^*\SR, \homo(V,W))$ the symbol of $\delta$, defined by the relation $\delta = \sigma \nabla + l$, where $l \in \homo(V,W)$ is the term of order $0$. Ellipticity of $\delta$ means that $\sigma(z)$ is an isomorphism when $z \neq 0$. Moreover, if $\sigma^*$ is the symbol of $\delta^*$, the following relation will be assumed: $\forall z \in \tg^*\SR, \sigma^*(z) \sigma(z) = \abs{z}^2 \Id_V$.

\par Fix $E>0$ and $\rho \in ]0,e^{-1}[$, the latter being small. Here, as in the rest of the text $\Pi_E$ is the projection on the space spanned by eigenfunctions of the Laplacian whose eigenvalue is bigger than $E$. This well-known lemma will be of use in the upcoming estimates.
\begin{lema} \label{estiml2}
Let $E>0$ and $\eta \in \cinf(V)$ be given. $\exists c_1 >0$ such that there exists an unique $u \in (\Pi_E \ldeu(V)) \cap \mathsf{W}^{^{2,2}}(V)$ satisfying $\nabla^* \nabla u = \Pi_E \eta$. Moreover, $u \in \cinf(V)$ and
\[
\nr{\nabla u}_{\ldeu}^2 + E \nr{u}_{\ldeu}^2 \leq (1 + \frac{c_1}{E}) \abs{\bint{M}{} \gen{u,\eta} }.
\]
\end{lema}
Suppose that $\chi \in \cinf(V)$ is orthogonal to the eigenspaces corresponding to small eigenvalues of the Laplacian, \ie $(1-\Pi_E)\chi =0 $. It will frequently be decomposed as:
\[
\chi = q + b_1 \nabla b_2
\]
where $b_2$ is a section of a vector bundle $Y \to \SR$, $b_1$ a section of $\cinf(\homo (Y \otimes \tg^*, V) )$, and $q \in \cinf(V)$. 
\begin{prop} \label{Tt4.8}
Let $E$ and $\rho$ be as above. $\exists c_3(\diam \SR)$ and $c_4(vol \SR, \diam \SR)$ two real positive numbers so that given $\chi= q + b_1 \nabla b_2$ as above and for $u\in  \cinf(V)$ a solution of $\delta^* \delta u = \chi$, then
\[
\begin{array}{crcl}
(a) \quad & \nr{u}_{\jo{L}^o,\rho} 
             &\leq& 
                c_3 (\rho^{-1}|\ln \rho| \nr{u}_{\ldeu} + \nr{q}_{*,\rho} + \nr{b_1}_{\jo{L}^0,\rho} \nr{b_2}_{\jo{H},\rho} ).
\end{array}
\]
If moreover $(1-\Pi_E) \chi =0 $, and let $\kappa_1(\rho,E) = (1+\rho^{-4} E^{-1})$, there exists a unique solution and
\[
\begin{array}{crcl}
(b) \quad & \nr{u}_{\jo{L}^o,\rho} 
             &\leq& 
                c_4 \kappa_1(\rho,E) (\nr{q}_{*,\rho} + \nr{b_1}_{\jo{L}^0,\rho} \nr{b_2}_{\jo{H},\rho} ). 
\end{array}
\]
\end{prop}
\begin{proof}
Introduce a smooth function $\alpha:[0,\infty) \to [0,1]$ equal to $1$ on $[0,1]$ and $0$ on $[2,\infty)$. For a fixed $x$, this function enables to define a function which is constant on $B_\rho(x)$ and with support in $B_{2\rho}(x)$:
\[
\alpha_x(y) = \alpha(\rho^{-1} d(x,y)).
\]
The equality
\[
\nabla^* \nabla \abs{u}^2 = 2 \gen{u, \nabla^* \nabla u}_{g} - 2 \abs{\nabla u}^2
\]
allows, together with
\[
\delta^* \delta u  = \nabla^* \nabla u + \sigma' \nabla u + R u
\]
which comes from the relation  $\sigma^*(\cdot) \sigma (\cdot) = \abs{\cdot}^2$ satisfied by the symbol $\sigma$ of $\delta$, to write $\gen{u,\delta^* \delta u}_{g} = \gen{u,\chi}_{g}$ as
\[
\frac{1}{2}(\nabla^* \nabla \abs{u}^2 +\abs{u}^2)+ \abs{\nabla u}^2 + \gen{u,\sigma' \nabla u} + \gen{u, R u-\frac{1}{2} u} = \gen{u,\chi}.
\]
Both sides of this equality are then multiplied by $\alpha_x(\cdot) G(x,\cdot) $ then integrated over $\SR$. Here is what the first term gives:
\[
\begin{array}{l}
\bint{\SR}{} \alpha_x(\cdot) G(x,\cdot) (\nabla^* \nabla \abs{u(\cdot)}^2 +\abs{u}^2) \\
\qquad  = \bint{\SR}{} G(x,\cdot) (\nabla^* \nabla \abs{u}^2 +\abs{u}^2) -\bint{\SR}{} (1-\alpha_x(\cdot)) G(x,\cdot) (\nabla^* \nabla \abs{u}^2 +\abs{u}^2)\\
\qquad  = \abs{u(x)}^2 -\bint{\SR}{} (1-\alpha_x(\cdot)) G(x,\cdot)\abs{u(\cdot)}^2  - \bint{\SR}{} (1-\alpha_x(\cdot)) G(x,\cdot) \nabla^* \nabla \abs{u(\cdot)}^2 \\   
\qquad  = \abs{u(x)}^2 -\bint{\SR}{} (1-\alpha_x(\cdot)) G(x,\cdot)\abs{u(\cdot)}^2 - \bint{\SR}{} \abs{u(\cdot)}^2 \nabla^* \nabla [ (1-\alpha_x(\cdot)) G(x,\cdot)].  
\end{array}
\]
Thanks to \ref{difgreen}, for a constant $c_2$, $\abs{ \nabla^* \nabla [ (1-\alpha_x(y)) G(x,y)]}$ is bounded above by $K_1 \rho^{-2} \abs{\ln \rho}$ when $y \in B_{2\rho}(x) \setminus B_\rho(x)$ and zero elsewhere. It then follows that
\[ \eqtag \label{prest}
\begin{array}{l}
\abs{u(x)}^2 + \bint{B_\rho(x)}{} \abs{\nabla u(\cdot)}^2 \ln (d(x,\cdot)^{-1}) \\
  \qquad \qquad \leq \abs{u(x)}^2 + \bint{\SR}{} \abs{\nabla u(\cdot)}^2 \alpha_x(\cdot) G(x,\cdot) \\
  \qquad \qquad \leq  K_2 \bigg( \bint{\SR}{} (1-\alpha_x(\cdot))G(x,\cdot)\abs{u(\cdot)}^2 + \rho^{-2} \abs{\ln \rho} \bint{A_{\rho,2 \rho}}{} \abs{u(\cdot)}^2 \\
  \qquad \qquad \qquad \qquad + K_3 \bint{B_{2 \rho}}{} \abs{u(\cdot)}^2 \ln(d(x,\cdot)^{-1}) + K_4 \bint{B_{2 \rho}}{} \alpha_x(\cdot) G(x,\cdot) \abs{u} \abs{ \nabla u} \\
  \qquad \qquad \qquad \qquad + \bint{B_{2 \rho}}{}\alpha_x(\cdot) G(x,\cdot) \gen{u,\chi} \bigg) 
\end{array}
\]
The $\sup$ of the left-hand term on $x \in M$ bounds $\frac{1}{2} \nr{u}_{\jo{L}^0,\rho}^2$; thus in our bounds of the right-hand terms, a factor of $\nr{u}_{\jo{L}^0,\rho}$ will always have to be present. Each of the five term on the right-hand side of \eqref{prest} will be treated differently. 

\par {\it First term. } The integrand is of support in $\SR \setminus B_\rho(x)$, a rough bound allows us to rewrite it in a shape close to that of the second term, that is,
\[
\bint{\SR}{} (1-\alpha_x(\cdot))G(x,\cdot)\abs{u(\cdot)}^2 \leq \nr{u}_{L^{2}}^2 \nr{(1-\alpha_x(\cdot)) G(x,\cdot)}_{L^{\infty}},
\]
and, since $\nr{(1-\alpha_x(\cdot)) G(x,\cdot)}_{L^{\infty}} < K_5 \abs{\ln \rho}$, this term is bounded (if $(a)$ is to be shown) by $K_5 |\ln \rho| \nr{u}_{\linf} \nr{u}_{\ldeu}$ ($K_5$ depends on $\diam \SR$ and $vol \SR $). As for the bound that gives $(b)$, the work is to be done as in the treatment of the second term (which is done immediately below).
\par {\it Second term. } To obtain $(a)$, it suffices to notice that $\nr{u}_{\ldeu(A_{\rho,2\rho})} \leq \sqrt{3\pi} \rho \nr{u}_{\linf}$. Thus, the first and second term are bounded by $(K_5 +\sqrt{3\pi}\rho^{-1}) |\ln \rho| \nr{u}_{\linf} \nr{u}_{\ldeu}$.
\par However, in order to get $(b)$, first write, thanks to \ref{estiml2},
\[
\nr{u}_{\ldeu}^2 \leq c_1 E^{-1} \bint{\SR}{} \gen{u,\chi} \dd y.
\]
That last term, after decomposing $\chi$ and integration by parts, is bounded by
\[
\bint{\SR}{} \gen{u,\chi} \dd y \leq \nr{u}_{\linf} (\nr{q}_{\lun} + \nr{\nabla b_1}_{\ldeu} \nr{b_2}_{\ldeu}) + \nr{\nabla u}_{\ldeu} \nr{b_1}_{\linf} \nr{b_2}_{\ldeu}
\]
Covering $\SR$ by $K_6 \rho^{-2}$ balls, where $K_6$ is function of the volume of $\SR$, the $\lpe$ norms are bounded by Taubes norm:
\[
\begin{array}{rcl}
\nr{\cdot}_{\lun} &\leq & K_6 \rho^{-2} |\ln \rho|^{-1} \nr{\cdot}_{*,\rho}\\
\nr{\cdot}_{\ldeu} &\leq & \sqrt{K_6} \rho^{-1} |\ln \rho|^{-1/2} \nr{\cdot}_{2*,\rho}
\end{array}
\]
Finally, these inequalities give
\[
\rho^{-2} |\ln \rho| \nr{u}_{\ldeu}^2 \leq K_6 \rho^{-4} E^{-1} \nr{u}_{\jo{L}^0} (\nr{q}_{*,\rho} + \nr{b_1}_{\jo{L}^0,\rho} \nr{b_2}_{2*,\rho})
\]
\indent {\it Third term. } This one is bounded quite simply, as the singularity is integrable:
\[
\bint{B_{2 \rho}}{} \abs{u(\cdot)}^2 \ln(d(x,\cdot)^{-1}) \leq 8 \rho^{2} \abs{\ln \rho} \nr{u}_{\linf}^2 \leq 8 \rho^{2} \abs{\ln \rho} \nr{u}_{\jo{L}^0}^2
\]
This term is thus destined to disappear: for $\rho$ small enough, it can be subtracted from both sides of the inequality.
\par {\it Fourth term. } As the preceding one, this term will only be negligible for $\rho$ small. Bound it by 
\[
\begin{array}{rl}
\bint{B_{2 \rho}}{} \alpha_x(\cdot) G(x,\cdot) \abs{u} \abs{ \nabla u}  
      & \leq \nr{u}_{\linf} \nr{\nabla u}_{2*,\rho} \bigg( \bint{B_{2 \rho}}{}  \big( \alpha_x(\cdot) G(x,\cdot) \big)^2 / \ln (d(x,\cdot)^{-1}) \bigg)^{1/2} \\
      &\leq K_7 \rho |\ln \rho|^{1/2} \nr{u}_{\linf} \nr{\nabla u}_{2*,\rho} \\
      &\leq K_7 \rho |\ln \rho|^{1/2} \nr{u}_{\jo{L}^0}^2,
\end{array}
\]
where $K_7$ does not depend on the cut-off function since $\supp{x \in \SR} \nr{G(x,\cdot) / |\ln d(x,\cdot) |}_{\linf(B_{2\rho}(x))} \leq (2\pi)^{-1}+4c_2 \rho^2|\ln 2\rho|$.
\par {\it Last term. } First decompose $\chi = q + b_1 \cdot \nabla b_2$. The part containing $q$ is bounded simply thanks to lemma \ref{difgreen} by $c_2 \nr{u}_{\linf} \nr{q}_{*,\rho} $. The rest requires more care. First, integrate by parts:
\[
\begin{array}{rl}
\bint{B_{2 \rho}}{}\alpha_x(\cdot) G(x,\cdot) \gen{u,b_1 \nabla b_2}
         &= - \bint{B_{2\rho}}{} \bigg[ \alpha_x(\cdot) G(x,\cdot) \big( (\nabla u, b_1 \cdot b_2) + (u,\nabla b_1 \cdot b_2) \big) \\
         & \qquad \qquad + (\dd(\alpha_x(\cdot) G(x,\cdot)) \otimes u,b_1\cdot b_2) \bigg].
\end{array}
\]
Apart from the last term, lemma \ref{difgreen} and proposition \ref{l0ssmul} ($\nr{ab}_{*,\rho} \leq \nr{a}_{2*,\rho} \nr{b}_{2*,\rho}$) allows us to bound this by 
\[
c_2 \nr{u}_{\jo{L}^0,\rho} \nr{b_1}_{\jo{L}^0,\rho} \nr{b_2}_{2*,\rho}.
\]
As for the ultimate remaining term, use again lemma \ref{difgreen} to bound the difference between $ \dd(\alpha_x(\cdot) G(x,\cdot)) $ and $(2 \pi )^{-1} d(x,\cdot) \nabla d(x,\cdot)$. However, $\phi := \nabla d(x,\cdot) \in \cinf(S^1;\rr^2)$ whence the following bound is found for this remaining term:
\[
K_8 (\nr{u}_{\jo{L}^0,\rho} \nr{b_1}_{\jo{L}^0,\rho} \nr{b_2}_{2*,\rho} + \nr{u \otimes b_1}_{\jo{L}^0,\rho} \nr{b_2}_{\jo{L}^1,\rho})
\]
Using \ref{l0ssmul} yields:
\[
\bint{B_{2 \rho}}{}\alpha_x(\cdot) G(x,\cdot) \gen{u,\chi} \leq K \nr{u}_{\jo{L}^0,\rho} \nr{b_1}_{\jo{L}^0,\rho} \nr{b_2}_{\jo{H},\rho}
\]
The bounds found for the five terms enables (when $2K_7 \rho^2 |\ln\rho| < 1/2$ so that the third and fourth terms do not weight on the right-hand side) to show that 
\[
\nr{u}_{\jo{L}^0,\rho} ^2 \leq c_4  \nr{u}_{\jo{L}^0,\rho} (1+\rho^{-4} E^{-1}) (\nr{q}_{*,\rho} + \nr{b_1}_{\jo{L}^0,\rho} \nr{b_2}_{\jo{H},\rho} ). \qedhere
\]
\end{proof}
When $\delta$ is without order $0$ term, if $\eta \in \cinf(W)$ and $u \in \wud(V)$ are related by
\[
\delta u = \eta,
\]
The results of \ref{Tt4.8} apply using that $\delta^* \delta u = \delta^* \eta $. Indeed, since $\delta^* = - \sigma^* \nabla + l^*$, it suffices to take $q=l^*\eta$, $b_1 = -\sigma^* $ and $b_2 = \eta$ so as to have the following corollary.
\begin{cora}
  Let $\rho$ be a small positive number and $E>0$. Let $c_4>0$ as above, if $\eta \in \cinf(W)$ and $u \in \cinf(V)$ are so that $\delta u = \eta $, then
\[
\nr{u}_{\jo{L}^0,\rho} \leq c_3 (\rho^{-1} |\ln \rho| \nr{u}_{\ldeu} + \nr{\eta}_{\jo{H},\rho}) \leq c_4 (1+ \rho^{-6} E^{-1}) \nr{\eta}_{\jo{H}}.
\]
\end{cora}
If $\eta \equiv 0 $, it is still possible to get a bound on the norm of $u$, using standard results.
\begin{lema}
$\forall k \in \nn \exists c_{5,k}$ such that $\xi \in \cinf(V)$ and $\xi \in \ker \delta $, \ie $\delta \xi =0$, then
\[
\pnr{\nabla^{\otimes k} \xi }_{\linf} \leq c_{5,k} \nr{\xi}_{\ldeu}.
\]
\end{lema}

\subsection{Estimating the $\jo{L}^1$ norm.}
Information will now be obtained on the $\jo{L}^1$ norm of the solutions of the equation $\delta^* \eta = \chi $, with $\delta^*$ elliptic and again the decomposition of $\chi$ as $q + b_1 \nabla b_2$.
\begin{lema} \label{Tt4.12}
Let $\rho \in ]0,1[$ be sufficiently small, $\exists c_6 >0$ such that if $\delta^* \eta = \chi $ then
\[
\begin{array}{rl}
\nr{\eta}_{\jo{L}^1,\rho} 
     &\leq c_6 (\nr{q}_{*,\rho} 
      + |\ln \rho|^{1/2} \nr{b_1}_{\jo{L}^0,\rho}\nr{b_2}_{\jo{H},\rho} 
      + |\ln \rho|^{1/2} \nr{\eta}_{2*,\rho})\\
\end{array}
\]
\end{lema}
\indent The proof, being far from obvious, requires a preparatory lemma and a few extra notations. A description of a test function (which will be multiplied to the equation $\delta^* \eta = \chi$ in order to conclude by integration by parts) has to be done first. 
\par Let $V_0$ and $W_0$ be vector spaces of equal dimensions and let $ \sigma_0 \in \homo(\rr^2,\homo(V_0,W_0))$ be such that $\sigma_0(z)$ is an isomorphism $\forall 0\neq z \in \rr^2 $. Recall that $ \sigma^*_0 \in \homo(\rr^2,\homo(W_0,V_0))$, thus for $z \in \rr^2$, $\sigma_0^*(z) \sigma_0(z) \in \endo(V_0)$. Suppose further that $\sigma_0^*(z) \sigma_0(z) = \abs{z}^2 \Id$.
\par Let $\nabla_0$ be the Euclidean covariant derivative in $\rr^2$, then $\delta_0 = \sigma_0(\nabla_0)$ is an elliptic operator of order $1$ on $\rr^2$ which sends maps with value in $V_0$ to maps with value in $W_0$. Similarly, it sends sections of $(V_0 \otimes W_0) $ on sections of $(W_0 \otimes W_0)$.
\par Finally, since $W_0$ is an Euclidean  vector space, $\endo(W_0) $ identifies to $(W_0 \otimes W_0)$, and $1 \in (W_0 \otimes W_0)$ will mean identity as an endomorphism.
\begin{lema}\label{delsol}
Let $\psi \in \cinf(S^1)$, be seen as function on $\rr^2 \setminus \{0\}$ which is radially constant. $\exists t_1 \in \cinf(V_0 \otimes W_0 )|_{S^1}$ unique (seen as a section independent of the norm) and $t_2 \in V_0 \otimes W_0$ such that for $s(\cdot) = t_1(\cdot)+ t_2 \ln \abs{\cdot}$
\[
\delta_0 \left( s  \right) = \fr{\psi \otimes 1}{\abs{\cdot}}
\]
and, for $c_7$ a  universal constant
\[
\abs{t_2} + \nr{t_1}_{\linf(S^1)} + \nr{t_1}_{\mathsf{W}^{\!^{3,2}}(S^1)} \leq c_7 \nr{\psi}_{\ldeu(S^1)}
\]
\end{lema}
\begin{proof}
The operator $\delta_0$ has a (Green's) kernel defined by
\[
\id{p}_p(\cdot) = K_1 \fr{\sigma_0^*(y-p)}{\abs{y-p}^2} \in \homo(W_0,V_0)
\]
Let $\hat{y} = y/ |y|$. Let $\psi_L(y) = \fr{\sigma_0(\hat{y})}{2\pi} \bint{S^1}{} \sigma_0^*(\hat{x}) \psi(\hat{x}) \dd \hat{x}$ be a section of $W_0$ on $\rr^2 \setminus \{O\}$. Then $t_2 = \sigma_0^*(\hat{y}) \psi_L(\hat{y}) = \fr{1}{2\pi} \bint{S^1}{} \sigma_0^*(\hat{x}) \psi(\hat{x}) \dd \hat{x}$ is an element of $V_0$. Let $\psi_N =\psi-\psi_L$. A formal solution to the equation can be written as
\[
t_2 \ln \abs{p} + K_1
      \int_{\rr^2} \fr{\sigma_0^*(y-p)}{\abs{y-p}^2} \fr{\psi_N(\hat{y})}{\abs{y}} \dd y.
\]
Let $t_1(p)$ be the expression corresponding to the integral. If $\hat{p} = p / \abs{p}$ and by making a change of variables $y \to |p|y$, it appears that $t_1(p) = t_1(\hat{p})$. Consequently, if it converges, the integral defines a section on the circle. Let us now write $y$ in polar coordinates $(|y|,\hat{y})$, then 
\[
\begin{array}{rl}
\bint{S_1}{} \sigma_0^*(\hat{y}) \psi_N(\hat{y}) \dd \hat{y} 
      &= \bint{S_1}{} \sigma_0^*(\hat{y}) \psi(\hat{y}) \dd \hat{y} - \bint{S_1}{} \sigma_0^*(\hat{y}) \psi_L(\hat{y}) \dd \hat{y} \\
      &= \bint{S_1}{} \sigma_0^*(\hat{y}) \psi(\hat{y}) \dd \hat{y} - \bint{S_1}{} \fr{|\hat{y}|^2}{2\pi} \left( \bint{S^1}{} \sigma^*(\hat{x}) \psi(\hat{x}) \dd \hat{x}  \right) \dd \hat{y} \\
      &=0.
\end{array}
\]
Whence, using  $ \sigma_0^*(\hat{y}-p/|y|) =  \sigma_0^*(\hat{y}) - \sigma_0^*(p)/|y|$, 
\[
\begin{array}{rl}
t_1(p) &= \bint{\rr^2}{} \fr{|y|}{\abs{y-p}^2} \sigma_0^*(\hat{y}-p/|y|) \psi_N(\hat{y}) \dd |y| \dd \hat{y} \\
       &= \bint{\rr^2}{} \fr{|y|}{\abs{y-p}^2} \big(\sigma_0^*(\hat{y}) \psi_N(\hat{y}) - \sigma_0^*(p) \psi_N(\hat{y})/|y| \big) \dd |y| \dd\hat{y} \\
       &= \bint{\rr^2}{} \fr{|y|}{\abs{y-p}^2} \sigma_0^*(\hat{y}) \psi_N(\hat{y})  \dd |y| \dd\hat{y} + \bint{\rr^2}{} \fr{1}{\abs{y-p}^2} \sigma_0^*(-p) \psi_N(\hat{y}) \dd |y| \dd\hat{y}.
\end{array}
\]
Thus, the second integral is convergent. There remains to show that the first also converges. The eventuality of divergence could come from large values of $|y|$. Choose $p$ such that $|p|=1$, when $|y|>1$, the expansion 
\[
|y-p|^{-2} = |y|^{-2}(1 + \gen{\hat{y},p}/|y| + |p|^2/|y|^2)^{-1} = |y|^{-2} (1+O(|y|^{-1}))
\]
enables us to write
\[
\begin{array}{rl}
\bint{\rr^2 \setminus B_1(0)}{} \fr{|y|}{\abs{y-p}^2} \sigma_0^*(\hat{y}) \psi_N(\hat{y})  \dd |y| \dd\hat{y} 
       &= \bint{\rr^2 \setminus B_1(0)}{} |y|^{-1} \sigma_0^*(\hat{y}) \psi_N(\hat{y})  \dd |y| \dd\hat{y}  \\
       &\qquad \quad + \bint{\rr^2 \setminus B_1(0)}{} O(|y|^{-1}) |y|^{-1} \sigma_0^*(\hat{y}) \psi_N(\hat{y})  \dd |y| \dd\hat{y}.
\end{array}
\]
Integrating first on the angular coordinate, the first integral is shown to be zero, whereas the second converges. Thus the integral $t_1(p)$ is also convergent. 
\par The promised bounds on the norms of these function remain to be found. 
\[
|t_2| \leq (2\pi)^{-1} \nr{\sigma_0^*}_{\ldeu(S^1)} \nr{\psi}_{\ldeu(S^1)} \leq K_2 \nr{\psi}_{\ldeu(S^1)}.
\] 
As for $t_1$, it satisfies a first order ordinary differential equation, the norm of its derivative is bounded by that of $\psi$ (the difference between $\psi$ and $\psi_N$ is bounded by $\nr{\psi}_{\ldeu(S^1)}$). Thus, 
\[
\nr{\nabla t_1}_{\ldeu(S^1)} \leq K_2 \nr{\psi}_{\ldeu(S_1)}.
\]
By compactness of $S^1$, $\nr{t_1}_{\linf(S^1)} \leq K_3 \nr{\psi}_{\ldeu(S_1)}$ and consequently $\nr{t_1}_{\ldeu(S^1)} \leq \sqrt{2 \pi} K_3\nr{\psi}_{\ldeu(S_1)}$.
\end{proof}
\begin{proof}[Proof of lemma \ref{Tt4.12}: ]
Let $x \in \SR$ be fixed, $\rho < \injrad \SR$, and use a Gaussian coordinate system around $x$. The metric that comes up in the evaluations of the norms will be replaced by an Euclidean metric: indeed, the expressions $\bint{B_\rho(0)}{} \frac{ \langle v,\phi \otimes \eta \rangle_{g_E}}{\abs{\cdot}_{g_E}}$ and $\bint{B_\rho(x)}{} \frac{ \langle v,\phi \otimes \eta \rangle_{g}}{\abs{\cdot}_g}$ do not differ by much, the ratio between an Euclidean metric and the metric of $\SR$ is a power of $(1+ \rho^2)$. Since $\nr{v}_{\linf} \leq 1$ and $\nr{\phi}_{\ldeu(S^1)} \leq 1$, this difference is bounded by $K_1 \rho |\ln \rho|^{-1} \nr{\eta}_{2*,\rho}$ where $K_1$ bounds the absolute value of $\rho^{-1} \ln \rho \int_{0}{\rho} r^{-2} \big( (1+r^2)^k -1 \big) |\ln r|^{-1} \dd r$ for $\rho \in ]0,e^{-1}[$.
\par Gaussian coordinates give a local trivialization of the cotangent bundle $\tg^*|_{B_\rho(x)}$ by associating it to the cotangent bundle of $B_\rho(0) \subset \rr^2$. Let the local coordinates of the latter be written as $\dd y_i$, $i=1$ or $2$,and let $v = \sum v_i \otimes \dd y_i$. In a similar fashion, a local trivialization of $V|_{B_\rho(x)}$ and $W|_{B_\rho(x)}$ over $B_\rho(0) \times V_0 $ and $B_\rho(0) \times W_0$, where $V_0 = V|_x $ and $W_0 = W|_x$, is given by these local coordinates. 
\par Consider now $\sigma_0  = \sigma|_x$ where $\sigma$ is the principal symbol of the operator $\delta$. Then $\delta_0 = \sigma_0(\nabla_0)$ is defined as in lemma \ref{delsol}. This lemma applies on the components $\phi_1$ and $ \phi_2$ of $\phi$ to give two functions $s_1$ and $s_2$. Let $\id{s}$ be the section (in coordinates) of $V|_{B_\rho(x)}$ defined by
\[
\id{s} = s_1 \otimes v_1 + s_2 \otimes v_2.
\]
Multiplying both sides of the equation $\delta^* \eta = \chi$ by $\alpha_x \id{s}$ (where $\alpha_x$ is the cutoff function introduced before), an integration by parts reveals
\[ \eqtag \label{lien}
\bint{B_{2\rho}(x)}{} \gen{\delta (\alpha_x \id{s}) , \eta}_g = \bint{B_{2\rho}(x)}{} \gen{\alpha_x \id{s}, \chi}_g
\]
Decomposing $\delta = \delta_0 + d(x,\cdot) \delta'$, yields
\[
\delta \id{s} = \sum (\delta_0 s_i) \otimes v_i + \sum s_i \otimes \delta v + d(x,\cdot) \sum \delta' s_i \otimes v_i.
\]
Thus the left-hand side of \eqref{lien} can be rewritten as
\[
\begin{array}{rl}
\bint{B_{2\rho}(x)}{} \gen{\delta (\alpha_x \id{s}) , \eta}_g 
       &= \bint{A_{\rho, 2\rho}(x)}{} \gen{(\delta \alpha_x) \id{s}, \eta }_g + \bint{B_{2\rho}(x)}{} \alpha_x \gen{ \sum \frac{\phi_i v_i}{|\cdot|_{g_E}}, \eta }_g \\
       &\qquad + \bint{B_{2\rho}(x)}{} \gen{\alpha_x \id{s} \otimes \delta v, \eta}_g + \bint{B_{2\rho}(x)}{} \gen{\alpha_x d(x,\cdot) \delta' \id{s} \otimes v, \eta}_g
\end{array}
\]
As $\bint{B_{\rho}(x)}{} \gen{ \sum \frac{\phi_i v_i}{|\cdot|_{g_E}}, \eta }_g \leq \bint{B_{2\rho}(x)}{} \alpha_x \gen{ \sum \frac{\phi_i v_i}{|\cdot|_{g_E}}, \eta }_g$   In other words, the term whose bound is of interest is
\[
\begin{array}{rl}
\bint{B_{2\rho}(x)}{} \alpha_x \gen{ \sum \frac{\phi_i v_i}{|\cdot|_{g_E}}, \eta }_g 
       &= \bint{B_{2\rho}(x)}{} \gen{\alpha_x \id{s}, \chi}_g 
          - \bint{A_{\rho, 2\rho}(x)}{} \gen{(\delta \alpha_x) \id{s}, \eta }_g +  \\
       &\qquad - \bint{B_{2\rho}(x)}{} \gen{\alpha_x \id{s} \otimes \delta v, \eta}_g 
          - \bint{B_{2\rho}(x)}{} \gen{\alpha_x d(x,\cdot) \delta' \id{s} \otimes v, \eta}_g.
\end{array}
\]
Recall that $s_i(\cdot) = t_{1,i}(\cdot) + t_{2,i} \ln |\cdot|$. The last three terms are bounded as follows:
\[
\begin{array}{rl}
\abs{  \bint{A_{\rho, 2\rho}(x)}{} \gen{(\delta \alpha_x) \id{s} \otimes v, \eta }_g } 
     &\leq K_3 K_2 \nr{v}_{\linf} \big( \nr{t_1}_{\linf} \nr{\eta}_{\ldeu(B_{2\rho})}\\
     & \qquad + |\ln \rho|^{1/2} \abs{t_2}\nr{\eta}_{2*,2\rho} \big)\\
\abs{ \bint{B_{2\rho}(x)}{} \gen{\alpha_x s \otimes \delta v, \eta}_g }
     &\leq K_3 \nr{t_1}_{\linf}\nr{\nabla v}_{\ldeu(B_{2\rho})} \nr{\eta}_{\ldeu(B_{2\rho})} \\
     &\qquad  + K_3 \abs{t_2} \nr{\nabla v}_{2*,2\rho} \nr{\eta}_{2*,2\rho} \\
\abs{ \bint{B_{2\rho}(x)}{} \gen{\alpha_x d(x,\cdot) \delta' s \otimes v, \eta}_g }
     &\leq 4 \rho^2 K_3 K_4 \nr{v}_{\linf} \big( \rho \nr{\nabla t_1}_{\ldeu(S^1)} \nr{\eta}_{\ldeu(B_{2\rho})} \\
     & \qquad + 2 \pi \abs{t_2} \nr{\eta}_{\ldeu(B_{2\rho})} \big),
\end{array}
\]
where $K_2 = 2 \pi / \int_{\rho}^{2\rho} r^{-1} \dd r = 2\pi/ \ln 2$, $K_3$ depends on the symbol of $\delta$ and $K_4= \nr{\frac{d(x,\cdot)}{|\cdot|}} _{\linf(B_{2\rho})}$. Proposition \ref{l0ssmul} will be used to find the usual norms: $\nr{\cdot }_{\ldeu(B_{2\rho})} \leq |\ln \rho|^{-1/2} \nr{\cdot}_{2*,2\rho}$ given that $2\rho < e^{-1}$. 

\par Using $\chi = q + b_1 \nabla b_2$, the first term becomes
\[
\begin{array}{rl}
\bint{B_{2\rho}(x)}{} (\alpha_x \id{s}, \chi)_g
     &= \bint{B_{2\rho}(x)}{} \gen{\alpha_x \id{s}, q}_g + \bint{B_{2\rho}(x)}{} \gen{\alpha_x \id{s}, b_1 \nabla b_2}_g \\
     &= \bint{B_{2\rho}(x)}{} \gen{\alpha_x \id{s}, q}_g - \bint{B_{2\rho}(x)}{} \gen{\alpha_x \id{s}, (\nabla b_1) b_2}_g - \bint{B_{2\rho}(x)}{} \gen{\nabla (\alpha_x \id{s}), b_1 b_2}_g, 
\end{array}
\]
where an integration by parts took place in order to obtain the last line. The first of these three terms can simply be bounded by 
\[
\nr{t_1}_{\linf} \nr{v}_{\linf} \nr{q}_{\lun(B_{2\rho})} + \abs{t_2} \nr{v}_{\linf} \nr{q}_{*,2\rho}.
\] 
As for the second, it is bounded by 
\[
\nr{t_1}_{\linf} \nr{v}_{\linf} \nr{\nabla b_1}_{\ldeu(B_{2\rho})} \nr{b_2}_{\ldeu(B_{2\rho})}+ \abs{t_2} \nr{v}_{\linf} \nr{\nabla b_1}_{2*,2\rho} \nr{b_2}_{2*,2\rho}.
\]
The third can be written as:
\[
\begin{array}{rl}
\bint{B_{2\rho}(x)}{} (\nabla (\alpha_x \id{s}), b_1 b_2)_g
     &=  \bint{A_{\rho, 2\rho}(x)}{} \gen{ (\nabla \alpha_x) \id{s}, b_1 b_2}_g + \bint{B_{2\rho}(x)}{} \gen{\alpha_x (\nabla t_1 \otimes v), b_1 b_2}_g\\
     & \qquad + \bint{B_{2\rho}(x)}{} \gen{\alpha_x t_2 \frac{\nabla |\cdot|}{|\cdot|} \otimes v , b_1 b_2}_g + \bint{B_{2\rho}(x)}{} \gen{\alpha_x s \otimes \nabla v , b_1 b_2}_g.
\end{array}
\]
The bounds are obtained as follows:
\[
\begin{array}{rl}
\abs{ \bint{A_{\rho, 2\rho}(x)}{} \gen{ (\nabla \alpha_x) \id{s}, b_1 b_2}_g  } 
     &\leq K_2 \nr{v}_{\linf} \nr{b_1}_{\linf}  \big( \nr{t_1}_{\linf}  \nr{b_2}_{\ldeu(B_{2 \rho})}\\
     & \qquad + |\ln \rho|^{1/2} \abs{t_2} \nr{b_2}_{2*,2\rho} \big) \\
\abs{ \bint{B_{2\rho}(x)}{} \gen{\alpha_x \nabla t_1 \otimes v , b_1 b_2}_g } 
     &\leq 4 \nr{\nabla t_1}_{\ldeu(S^1)} \nr{v}_{\linf} \nr{b_1}_{\linf} \nr{b_2}_{\ldeu(B_{2\rho})} \\
\abs{ \bint{B_{2\rho}(x)}{} \gen{\alpha_x t_2 \frac{\nabla |\cdot|}{|\cdot|} \otimes v , b_1 b_2}_g }
     &\leq \nr{t_2}_{\linf} \nr{v}_{\linf} \nr{b_1}_{\linf} \nr{b_2}_{\jo{L}^1}\\
\abs{ \bint{B_{2\rho}(x)}{} \gen{\alpha_x s \otimes \nabla v , b_1 b_2}_g} 
     &\leq \nr{t_1}_{\linf} \nr{\nabla v}_{\ldeu(B_{2\rho})} \nr{b_1}_{\linf} \nr{b_2}_{\ldeu(B_{2\rho})}\\
     & \qquad + \abs{t_2} \nr{\nabla v}_{2*,2\rho} \nr{b_1}_{\linf} \nr{b_2}_{2*,2\rho}.
\end{array}
\]
Putting all these bound together yield lemma \ref{Tt4.12}. 
\end{proof}

\subsection{The kernel of $\Pi_E$.} 
This subsection provide bounds on the part that has so far been neglected. Our goal is to get a bound on $(1-\Pi_E)\chi$ in terms of the norms of $q $, $b_1$, and $b_2$. To alleviate notations, $\pi_E$ will denote the projection on small eigenvalues of $\nabla^* \nabla $: $\pi_E = 1 - \Pi_E $. Let $N(E)$ be the number of eigenvalues $\leq E$ and let $\{v_i\}_{i=1}^{N(E)}$ a basis of the image of $\pi_E$:
\[
\pi_E \chi = \somme{i=1}{N(E)} \left[\bint{\SR}{} \gen{v_i,\chi}_g \right] v_i.
\]
The main result of this section is to bound $\nr{\pi_E \chi}_{*,\rho}$ by $\nr{q}_{*,\rho'}$, $\nr{b_1}_{\jo{L}^0,\rho'}$ and $\nr{b_2}_{2*,\rho'}$ but with a parameter $\rho' \neq \rho$. But some preparatory lemmas have to be established first.
\begin{lema}\label{gradvp}
Let $\eps_2 \in \rr_{>0}$, there exists constants $c_{8,n}$ depending on $\eps_2$ and on the metric on $\SR$, such that for $v$ an eigenvector of $\nabla^*\nabla$ whose eigenvalue is $\lambda$ and whose norm $\nr{v}_{\ldeu}=1$
\[
\nr{\nabla^{\otimes n}v}_{\linf} \leq c_{8,n} \max(1, \lambda^{\frac{1}{2}+\eps_2})
\]
\end{lema}
\begin{proof}
Let $\eps = \frac{4\eps_2}{3+2\eps_2}$, so that $\frac{1+\eps}{2-\eps} = \frac{1}{2} + \eps_2$ Choose $\rho$ such that $\rho \leq \injrad \SR$ and $\rho^{-\eps} \geq |\ln \rho|$. When $n =0$, this bound is a consequence of \ref{Tt4.8}.$(a)$. Indeed, taking $\delta = \nabla$, $E<\lambda$, $u=v$ and $\chi = \lambda v$, yields
\[
\begin{array}{rl}
\nr{v}_{\linf} \leq \nr{v}_{\jo{L}^0} &\leq c_3 (\rho^{-1} |\ln \rho| \nr{v}_{\ldeu} + \lambda \nr{v}_{*,\rho}) \\ 
                                      &\leq c_3 (\rho^{-1}|\ln \rho| \nr{v}_{\ldeu} + \lambda \nr{v}_{\linf} K_1 \rho^{2} |\ln \rho|),
\end{array}
\]
where $K_1 \leq 8 \pi$ comes from the integral $\rho^{-2}|\ln \rho|^{-1} \int_{B_{2\rho}} |\ln r| \dd r$. Thus, under the condition that $\lambda K_1 \rho^{2} |\ln \rho| \leq \lambda K_1 \rho^{2-\eps} < 1/2$, then
\[
\nr{v}_{\linf} \leq  2 c_3 \rho^{-1}|\ln \rho| \leq 2 c_3 \rho^{-(1+\eps)} \leq 2 c_3 \max(K_2,(2 K_1 \lambda)^{(1+\eps)/(2-\eps)}).
\]
The last inequality is obtained by taking $\rho$ as large as allowed (so $K_2$ depends on $\injrad \Sigma$ and $\eps$). Induction may now be invoked. Suppose that the statement is true for any integer $\leq k$. Then, applying \ref{Tt4.8}.$(a)$ to $u= \nabla^2(\nabla^{\otimes k+1} v)$ and $\chi = \lambda \nabla^{\otimes k+1} v + \sum_{0\leq i \leq k} \jo{R}_i \nabla^{\otimes k-i} v_i$, the conclusion follows by the same argument (the $\jo{R}_i$ depend only on the metric).
\end{proof}
\begin{lema}\label{rgpie}
Let $N(E)$ be the rank of $\pi_E$, then $N(E) \leq c_9 (E+1)$
\end{lema}
This is Weyl's law, see (among many possibilities) \cite[p.204]{BGM} or \cite[Corollary 2.5, p.361]{Ura}.
\begin{lema}\label{surjres}
There exists a constant $E_0$ which depends on the metric such that $\forall x \in \SR$ if $r_x : \pi_{E_0} \cinf(V) \to V|_x$ is the restriction at $x$ and $r_x \circ \nabla : \pi_{E_0} \cinf(V) \to (\tg \otimes V)|_x$ is the restriction of the derivatives, then $r_x$ and $r_x \circ \nabla$ are surjective.
\end{lema}
\begin{proof}
Let $\id{s} $ be a smooth section of $V$ such that $\nr{\id{s}}_{\ldeu} =1$. Since $\SR$ is compact, $\nr{\nabla \id{s}}_{\ldeu} \leq K_1$. Thus, the expression of $\id{s}$ in terms of eigenfunctions converges point-wise. Thus, for any basis of $V|_x$ there exists a $E_x$ such that this basis can be approximated by elements of $\pi_{E_x}$. This surjectivity remains valid for points close to $x$, and by compactness of $\SR$ the conclusion is achieved. The same argument works for $r_x \circ \nabla$.
\end{proof}
The main result of this section is now at hand.
\begin{lema}\label{Tt4.14}
Let $\kappa_2(\rho',E) = (1+\rho'^2E^{5/3})$. There exists a constant $c_{10}$ such that for $E \in \rr_{>0}$, and $\rho, \rho' < R_{10}$, then for $\chi \in \cinf(V)$ which can be written as $\chi = q + b_1 \nabla b_2$
\[
\nr{\pi_E \chi}_{*,\rho} \leq c_{10} \frac{\rho^2 |\ln \rho|}{\rho'^2 |\ln \rho'|} \kappa_2(\rho',E) (\nr{q}_{*,\rho'} + \nr{b_1}_{\jo{L}^0,\rho'} \nr{b_2}_{2*,\rho'}).
\]
\end{lema}
\begin{proof}
For two integers $n,m$ big enough, it is possible to choose a set $\Omega$ such that
\begin{itemize}
\item $\union{x \in \Omega}{} B_{n\rho'}(x) = \SR $,
\item $B_{\rho'}(x) \cap B_{\rho'}(x') = \vide$ if $x \neq x'$ are two points of $\Omega$,
\item for $\Omega' \subset \Omega$,  $|\Omega'| \geq m \imp \inter{x \in \Omega'}{} B_{n\rho'}(x) = \vide$.
\end{itemize}
This set is easily realized in Euclidean space. Since $\SR$ can be isometrically embedded in $\rr^k$, this remains true up to a small perturbation. Consider again the cutoff function $\alpha_x$ defined this time with parameter $n\rho'$ rather than $\rho$. Furthermore let $\gamma_x(\cdot) = \alpha_x(\cdot)/ \sum_{y \in \Omega} \alpha_y(\cdot)$ be the partition of unity associated to the covering of $\SR$ by $\{ B_{n\rho'}(x) \}_{x \in \Omega}$. Moreover, the gradient of $\gamma_x$ behaves nicely: $|\nabla \gamma_x(\cdot)| \leq K_1 \rho'^{-1}$.
\par As the projection $\pi_E$ is a linear operator, the bound on $\chi$ can be obtained thanks to $\chi = \sum_{x \in \Omega} \gamma_x(\cdot) \chi(\cdot)$. Using lemma \ref{surjres}, for each point $x \in \Omega$, there exists a $\ldeu$-orthonormal basis $\{v_i\}_{i=1}^{N(E)}$ of $\pi_E \cinf(V)$ such that $v_i \in \cinf(V)$ and, when $i>N(E_0)$, $r_x v_i = 0= r_x \nabla v_i$. Again, upon integrating by parts, the following expression for the projection of $\chi$ on $v_i$ can be obtained
\[\eqtag \label{dechi}
\begin{array}{rl}
\bint{\SR}{} \gen{v_i,\gamma_x \chi}_g  
      &=     \bint{\SR}{} \gen{v_i,\gamma_x q}_g 
              - \bint{\SR}{} \gen{v_i, \gamma_x (\nabla b_1) b_2}  \\
      &\qquad - \bint{\SR}{} \gen{\nabla v_i, \gamma_x b_1 b_2}
              - \bint{\SR}{} \gen{v_i, (\nabla \gamma_x) b_1 b_2}.
\end{array}
\]
Consider the projection of $\pi_E \gamma_x \chi$ on $v_i$ when $i \leq N(E_0)$. In that case, lemma \ref{gradvp} enables us to bound $v_i$ and $\nabla v_i$ uniformly by $c_{8,1} (1+E_0)^{2/3}$, thus the right-hand terms in \eqref{dechi} are bounded respectively by 
\[
\begin{array}{l}
|\ln \rho'|^{-1} \nr{q}_{*,2n\rho'} + |\ln \rho'|^{-1}\nr{\nabla b_1}_{2*,2n\rho'} \nr{b_2}_{2*,2n\rho'} \\
\qquad + \rho' |\ln \rho'|^{-1} \nr{b_1}_{\linf} \nr{b_2}_{2*,2n\rho'} + K_1^2 |\ln \rho'|^{-1} \nr{b_1}_{\linf} \nr{b_2}_{2*,2n\rho'}.
\end{array}
\]
All these norms can be put together to give 
\[
\abs{\bint{\SR}{} \gen{v_i,\gamma_x \chi}_g} \leq |\ln \rho'|^{-1} K_2 (\nr{q}_{*,\rho'} + \nr{b_1}_{\jo{L}^0,\rho'} \nr{b_2}_{2*,\rho'}),
\]
where \ref{l0ssmul} is used to pass from the parameter $2n\rho'$ to $\rho'$. Also, $\nr{v_i}_{*,\rho} \leq c_{8,0} (1+E_0)^{2/3} \rho^2 |\ln \rho|$, which yields:
\[
\nr{v_i \bint{\SR}{} \gen{v_i,\gamma_x \chi}_g}_{*,\rho} \leq K_3 \rho^2 |\ln \rho| |\ln \rho'|^{-1} (\nr{q}_{*,\rho'} + \nr{b_1}_{\jo{L}^0,\rho'} \nr{b_2}_{2*,\rho'}).
\]
Now, if $i > N(E_0)$ the choice of the $v_i$ gives that $|v_i| \leq c_{8,2} \rho'^2 E^{2/3}$ and $|\nabla v_i| \leq c_{8,2} \rho' E^{2/3}$. This time the right-hand terms of \eqref{dechi} are bounded as follows:
\[
\begin{array}{l}
c_{8,2} E^{2/3} \rho'^2 |\ln \rho'|^{-1} \big[ \nr{q}_{*,2n\rho'} + \nr{\nabla b_1}_{2*,2n\rho'} \nr{b_2}_{2*,2n\rho'} \\
\qquad \qquad \qquad \qquad + \nr{b_1}_{\linf} \nr{b_2}_{2*,2n\rho'} + K_1^2 \nr{b_1}_{\linf} \nr{b_2}_{2*,2n\rho'} \big].
\end{array}
\]
Since $\nr{v_i}_{*,\rho }\leq c_{8,0} (1+E)^{2/3} \rho^2 |\ln \rho|$, then
\[
\nr{v_i \bint{\SR}{} \gen{v_i,\gamma_x \chi}_g}_{*,\rho} \leq K_4 \rho^2 |\ln \rho| r^2 |\ln \rho'|^{-1} (\nr{q}_{*,\rho'} + \nr{b_1}_{\jo{L}^0,\rho'} \nr{b_2}_{2*,\rho'}).
\]
As $N(E)-N(E_0) \leq K_5 E$, the decomposition $\pi_E \gamma_x \chi = \somme{x \in \Omega}{} v_i \bint{\SR}{} \gen{v_i,\gamma_x \chi}_g $ enables to conclude that
\[
\nr{\pi_E \gamma_x \chi}_{*,\rho} \leq K_6  \rho^2 |\ln \rho| |\ln \rho'|^{-1} (1+ \rho'^2 E^{5/3}) (\nr{q}_{*,\rho'} + \nr{b_1}_{\jo{L}^0,\rho'} \nr{b_2}_{2*,\rho'}).
\]
The finishing touch consists in noticing that the cardinality of $\Omega$ is bounded by $K_7 \rho'^{-2}$, where $K_7$ depends on the volume.
\end{proof}

\subsection{Existence and \emph{a priori} bound on solutions} \label{Ts5}
It will be convenient to introduce
\[ \eqtag \label{croch}
\gen{\chi }_\rho = \nr{q}_{*,\rho} + \nr{b_1}_{\jo{L}^0,\rho}\nr{b_2}_{\jo{H},\rho}
\]
\par The linearized operator of $\db_J$ at $f$ is the operator $D_f$ introduced in McDuff and Salamon's book \cite[\S{}3]{mds1}. Even if for many structures it is invertible when $f$ is $J$-holomorphic, the present situation requires to look at this operator for a function which is precisely not $J$-holomorphic (at least in the complement of $K$, the compact set where the approximation is to be made). The projection $\Pi_E$ enables to avoid problems that arise from a lack of surjectivity.
\par Define first $\chi'(u)$ by
\[
\delta^* \delta u  = \nabla^* \nabla u + \sigma' \nabla u + R u = \nabla^* \nabla u + \chi'(u),
\]
where $\sigma'$ is the symbol of a first-order operator. A wise use of lemma \ref{estiml2} will give the existence of a  $u \in \cinf(V)$ such that $\Pi_E \delta^* \delta u = \Pi_E \chi$.  
\begin{lema}
Let $\delta$ be an elliptic operator as above, there exists a constant $c_{11}$ (which depends on $\delta$) such that when $E> c_{11}$, the equation $\Pi_E \delta^* \delta u = \Pi_E \chi$ admits a unique solution $u \in \Pi_E \cinf(V)$. Moreover this solution depends continuously and linearly on $\chi$.
\end{lema}
\begin{proof}
Write $\nabla^* \nabla u  = \Pi_E (\chi -\chi'(u))$. Lemma \ref{estiml2} insures the existence of a section $u_\chi$ such that $\nabla^* \nabla u_\chi = \Pi_E \chi$ and of $\psi(u)$ solution to $\nabla^* \nabla \psi(u) = - \Pi_E \chi'(u)$. Thus, the problem can be expressed as the existence of a fixed point for
\[
u = \psi(u) + u_\chi.
\]
It suffices to show that $u \mapsto \psi(u)$ is contracting as a map from $\Pi_E \wud(V)$ to itself. First, since $\nr{\chi'}_{\ldeu} \leq K_1 \nr{u}_{W^{1,2}}$ lemma \ref{estiml2} shows that if $E > c_1$
\[
\begin{array}{rrll}
     &E\nr{\psi(u)}_{\ldeu}^2 &\leq 2 \abs{\int \gen{\psi(u),\chi'} }
                            &\leq 2 K_1 \nr{\psi(u)}_{\ldeu} \nr{u}_{\wud} \\
\imp & \nr{\psi(u)}_{\ldeu}   &\leq 2 K_1 E^{-1} \nr{u}_{\wud}.
\end{array}
\]
Using this inequality, a second application of the same lemma gives 
\[
\begin{array}{rl}
\nr{\nabla \psi(u)}_{\ldeu}^2 &\leq 2 \abs{\int \gen{\psi(u),\chi'} } \\
                            &\leq 2 K_1 \nr{\psi(u)}_{\ldeu} \nr{u}_{\wud} \\
                            &\leq 4 K_1^2 E^{-1} \nr{u}_{\wud}^2.
\end{array}
\]
Thus, $\nr{\psi(u)}_{\wud} \leq 4 K_1 E^{-1/2} \nr{u}_{\wud}$, that is the linear map $\psi: \Pi_E \wud(V) \to \Pi_E \wud(V)$ in question is contracting given that $E > \max(16K_1^2,1,c_1)$. In other words, $u=\psi(u) + u_\chi \ssi (\Id - \psi)(u) = u_\chi$ . However $\Id-\psi$ can be inverted using power series (which converges since $\nr{\psi} < 1$). The solution to our fixed point equation is $u=(\Id - \psi)^{-1}(u_\chi)$. Thus, linearity of the dependence on $u$ comes from the linear dependence of $u_\chi$ on $\chi$. Arguments of ellipticity enables us to conclude that $u\in \Pi_E \cinf(V)$.
\end{proof}
\begin{teoa} \label{exborn}
Let $E$ and $\rho$ be positive numbers. The equation $\Pi_E \delta^* \delta u = \Pi_E \chi$ admits a unique solution $u \in \Pi_E \cinf(V)$ which depends continuously and linearly on $\chi$ and such that 
\[
\nr{u}_{\jo{L},\rho}
   \leq c_{12} \kappa_1(E,\rho) |\ln \rho| \gen{\chi}_\rho
\]
where $\kappa_1(E,\rho) = 1+\frac{1}{E\rho^4}$.
\end{teoa}
\begin{proof}
The previous lemma covers all the assertions of the theorem with the exception of the bound on $\nr{u}_{\jo{L}}$. This is done using lemma \ref{Tt4.8}:
\[ \eqtag \label{norml0}
\nr{u}_{\jo{L}^0,\rho} \leq c_4(1+ E^{-1}\rho^{-4} |\ln \rho|) \gen{\chi}_\rho
\]
The $\jo{L}^1$ norm of $\nabla u$ requires more work. First observe that $u$ satisfies the following system of equations
\[
\begin{array}{rl}
\nabla^* \nabla u &= \Pi_E (\chi + \chi'(u))\\
\nabla \nabla u &= R^{\nabla} u
\end{array}
\]
where $R^{\nabla}$ is the curvature tensor. The operator $\nabla^* \oplus \nabla : \cinf(V) \to \cinf(V) \oplus \cinf(\tg^*\SR \times \tg^*\SR \times V)$ is elliptic of the first order. Lemma \ref{Tt4.12} can be used on $(\nabla^* \oplus \nabla) (\nabla u) = \Pi_E (\chi + \chi'(u)) \oplus R^{\nabla} u$ to get that 
\[
\nr{\nabla u}_{\jo{L}^1}
   \leq c_6 \Big( \nr{q}_{*,\rho} 
         + \nr{ \Pi_E \chi'(u)}_{*,\rho} 
         + \pnr{R^\nabla u}_{*,\rho} 
         + |\ln \rho|^{1/2} (\nr{b_1}_{\jo{L}^0,\rho} \nr{b_2}_{\jo{H},\rho} 
         + \nr{\nabla u}_{2*,\rho} ) 
          \Big)
\]
where $q$, $b_1$ and $b_2$ come from the decomposition $\Pi_E \chi = q + b_1 \nabla b_2$. For a constant $K_1$ which depends of the terms of order less than  $2$ in $\delta^* \delta$, 
\[
\nr{ \Pi_E \chi'(u)}_{*,\rho} \leq K_1 (\nr{u}_{*,\rho} + \nr{\nabla u}_{2*,\rho}) \leq K_1 \nr{u}_{\jo{L}^0,\rho}.
\]
Moreover, there exists another constant such that $\nr{R^\nabla u}_{*,\rho} \leq K_2 \nr{u}_{*,\rho}$. Thus,
\[
\nr{\nabla u}_{\jo{L}^1,\rho} \leq K_3 \big( \nr{q}_{*,\rho}
  + |\ln \rho|^{1/2} ( \nr{b_1}_{\jo{L}^0,\rho} \nr{b_2}_{\jo{H},\rho} + \nr{\nabla u}_{2*,\rho})
  + \nr{u}_{\jo{L}^0,\rho}
\]
Using \eqref{norml0} to get rid of the terms in $u$ and then adding the resulting inequality with \eqref{norml0} gives
\[
\nr{u}_{\jo{L},\rho} 
  \leq K_4 |\ln \rho|^{1/2} \big( \nr{q}_{*,\rho} + \nr{b_1}_{\jo{L}^0,\rho} \nr{b_2}_{\jo{H},\rho} \big) \qedhere
\]
\end{proof}
\indent Note that if for some reason the operator $\delta$ is surjective, it is no longer  necessary to project on large eigenvalues of the Laplacian. Thus, it is possible to obtain the same estimates. Here is a case of interest.
\begin{cora} \label{Tt5.7}
Let $\SR= \cp^1$, let $\delta$ be surjective, and let $\rho < e^{-1}$, and let $u$ be a solution of $\delta \delta^* u = \chi$, then
\[
\nr{u}_{\jo{L},\rho} \leq 3 c_{12} \rho^{-4} |\ln \rho|^{3/2} \big( \nr{q}_{*,\rho} + \nr{b_1}_{\jo{L}^0,\rho} \nr{b_2}_{\jo{H},\rho}  \big). 
\]
In particular, this inequality holds for $\rho = 10^{-1}$.
\end{cora}
\begin{proof}
The proof is identical to the one of the previous theorem with the exception that it is only required to take $E$ the smallest eigenvalue of the Laplacian. Fixing $E$ however cannot guarantee that $E^{-1}\rho^{-4}$ will be bounded, and $\rho$ must consequently also remain fixed. 
\end{proof}

\section{Realizing Newton's method} \label{chrun}
\indent We briefly recall the intuitive idea to tackle the problem and sketch the contents of this section. Given a (non-constant) $J$-holomorphic map $g_0:U \to M$ from some open set $U$ of a Riemann surface $\SR$ to an almost-complex manifold $(M,J)$, it will be extended (as there exists a $\czer$ extension) in a $\cinf$ fashion to a map $g:\SR \to M $ defined on the whole of $\SR$, $g$ being identical to $g_0$ on the compact $K \subset U$ where the approximation is to be done. There will then be a set, presumably quite big, where $g$ will not be $J$-holomorphic. To make this map $J$-holomorphic on a bigger set, its values on small discs will be replaced by those of $J$-holomorphic maps having local expansion close to that of $g$ on the boundary of these small discs. However, in order to keep the differential of the approximate solution $f$ bounded, it will also be required to change the metric of the surface (so that it metrically looks like the surface where many "connect summed" $\cp^1$). This process is described in subsection \ref{greffe}.
\par Once the approximate solution $f$ has been obtained, the linear equation must be solved (that is the inverse of the linear operator must be found) in order to apply Newton's method. This can unfortunately not done in one swoop. First, in order to deal with the metrically strange manifold that the many graftings have created, it will turn out more convenient to split the equation on each parts (the initial surface $\SR$ and the $\cp^1$ grafts) with, for the sake of consistency, some interaction between each other. Similar process are already present in McDuff and Salamon description of the gluing \cite[\S{}10.5]{mds1}, Taubes' work on anti-self-dual metrics \cite[\S{}6]{tau} and Donaldson work on instantons \cite[\S{}IV(iv)]{don2}. Subsection \ref{Ts6} is concerned with this splitting of the linear equation.
\par The inversion of the linear operator only takes place in subsection \ref{Ts7}. Though the equations have split they still interact between each other. First the resolution (and bounds) on the $\cp^1$ in terms of the normal data and the perturbation from the base $\SR$ is done. Likewise on the base $\SR$, but there are two problem. The first is that the small eigenvalues of the Laplacian must be taken out to insure inversion; the treatment of these small eigenvalues is postponed to subsection \ref{Ts9}. The second is that the perturbation coming from the $\cp^1$ depend on what happens in $\SR$. The result will be a (multi-)linear map $\xi_0 = A_1 \eta + A_2 \xi_0$ and a correct choice of parameter will make the norm of $A_2$ small so that $(\Id - A_2)$ is invertible. 
\par Once proper estimates for the inverse of the linearization have been made, the contraction as in \eqref{newt} is then proved in subsection \ref{Ts8}. This first fixed-point argument will yield an $E$-quasi-solution, a perturbation that would give a honest solution, if not for our negligence of the small eigenvalues. This section gives a map $H^{nl}_E$ from approximate solutions to a default of solution which lives in the space of small eigenvalues, $\img \pi_E$. 
\par Subsection \ref{Ts9} deals with the small eigenvalues. After some observations on Taubes' norm in $\img \pi_E$ and the term resulting from constructions by grafting, a family of approximate functions $f_\nu$ parametrized by $\nu \in \scr{B} \subset \img \pi_E$ is constructed. Relatively rough estimates enable to conclude, by a fixed point argument on $\nu \mapsto f_\nu \overset{H^{nl}_E} \to \img \pi_E$, that there is a honest solution.

\subsection{Grafting and the approximate solution} \label{greffe}

The initial data is the function $g_0$ defined on $U \subset \SR$ and to be approximated on a compact $K \subset U$. The first step is to extend in a $\cinf$ fashion $g_0 \bv_K$ to the whole of $\SR$ by a function $g$. $g$ can barely be expected to be pseudo-holomorphic outside $K$. Grafting many localized solutions to $g$ will give the approximate solution. 

\begin{lema} \label{locexp}
Suppose $J$ is Lipschitz. Take a point $z_0 \in \SR$, a holomorphic chart on $B_r(z_0) \subset \SR$ $\phi:B(z_0) \to \cc$ sending $z_0$ to $0$ and a chart $\Phi: B_R\big( g(z_0) \big) \to \rr^{2n}$ such that $\Phi^*J_{g(z_0)} =J_0 = \sm22{0}{-\un_{\rr^n}}{\un_{\rr^n}}{0}$ and $\Phi(g(z_0))=0$. Then there exist $a,b \in \rr^{2n}$ such that $\Phi \circ g \circ \phi^{-1} (z) = az+ b\zb +O(\abs{z}^2)$ where $\sqrt{-1} a := J_0 a \in \rr^{2n}$. 
\end{lema}
\begin{proof}
In the holomorphic case this is obvious, and since the structure $J$ is Lipschitz the deviations from the holomorphic case will remain of higher order; remark this is implicit in Sikorav's discussion of the local behavior, \cite{Sik}. 
\end{proof} 

Let us now clarify the one of the two main hypothesis of theorem \ref{runge}.

\begin{dfa}\label{ddtp}
$(M,J)$ has the \emph{double tangent property} if there is a dense set $E$ in the bundle $\tg M \oplus \tg M$ such that for any $(a,b) \in E_m \subset \tg_m M \oplus \tg_m M$ and $\forall \eps \in ]0,1/3[$, there exists a $r_0(a,b,\eps) \in \rr_{>0}$ such that for any $r \in ]0,r_0[$, there exists a pseudo-holomorphic map $\habr : \cp^1 \to (M,J)$ such that, in local charts (as in lemma \ref{locexp}), if $ \tfrac{r}{1+r^\eps} \leq |z| \leq r(1+r^\eps)$, $\Phi \circ \habr \circ \phi^{-1} (z) = a z + b r^2/z + O(r^{1+\eps})$. Furthermore, for fixed $\eps$ and $K \in \rr_{>0}$, 
\[
\sup {r_0(a,b,\eps) \mid (a,b) \in E, \pnr{a} + \pnr{b} \leq K} >0.                                                                                                                                                                                                                                                                                                                                                                                                                                                                                                                                                                                                                                                                                                                                                                                                                                                                                                                                                                                                                                                                                                                                                                                                             \]
\end{dfa}   
If, per chance, it happens that $\db_J g(z_0) =0$ (or, in particular, that $\dd g(z_0) =0$) at some point $z_0$ where the grafting is to be made, then it is possible to go for a simpler procedure. Indeed, in lemma \ref{locexp} $b=0$, so that, in those charts, replacing $g$ on a small ball by the function $az$ and gluing back outside the ball to the previous function (using cut-off functions) will turn out to give much nicer estimates than when $b \neq 0$ (see Donaldson's paper \cite[\S{}3]{don}, where a similar grafting procedure goes on more smoothly than here). 

Let us focus on the localized solution when $b \neq 0$ (and for $r$ sufficiently small). It will be $\habr$ the family of $J$-holomorphic curves coming from the double tangent property (see definition \ref{ddtp}). Recall the $\habr$ can be obtained as the result of the gluing process described in \cite[\S{}2]{moi-cyl}. Thus let
\[
\habr(z) = az+ b\frac{r^2}{z} + O(r^{1+\eps_0}),
\]
for $\eps_0 \in ]0,\frac{1}{3}[$ and $r<R_0(M,J,g)$. This approximation will be used with say $\eps_0 = \frac{33}{100}$. 

\begin{dfa}
A $J$-holomorphic graft on $g$ at $z_0$ of parameter $r$ (where $r<R_0$) is the function $g_r = g \sharp_{z_0} \habr$ defined, in local charts and for $a,b$ as in lemma \ref{locexp}, as follows:
  \begin{equation*}
g_r(z) = \left\{ \begin{array}{llrcl}
  g(z)                                             & \textrm{if } &r(1+r^\eps) <& \abs{z} & \\
  \beta(\abs{z})g(z) + (1-\beta(\abs{z})) \habr(z) & \textrm{if } &r          <& \abs{z} &< r(1+r^\eps) \\
  \habr(z)                                         & \textrm{if } &             & \abs{z} &<r          
  \end{array}
\right.
  \end{equation*}
where $\eps \in ]0,\eps_0[$ and
\[
\beta(z)= 1-\beta_{r(1+r^\eps),r} (z) = \left\{
\begin{array}{llrcl}
    1                                        & \textrm{if } & r(1+r^\eps) <& \abs{z} & \\
    \fr{\ln \abs{z} - \ln r}{\ln (1+r^\eps)}  & \textrm{if } & r          <& \abs{z} &< r(1+r^\eps) \\
    0                                        & \textrm{if } &             & \abs{z} &<r          
\end{array} \right.
\]
\end{dfa}
Let us dwell a bit on the domain $\abs{z} <r(1+r^\eps)$ in the above charts: the function will not be modified again there, and when $|z| < r $ it is actually $J$-holomorphic since $\habr$ is $J$-holomorphic. So $\db_J g_r$ is identically $0$ on $|z| < r $. A bit more information is required out of this grafting procedure.
\begin{lema}
Let $g_r$ be the above $J$-holomorphic graft on $g$, let $\SR \sharp_{z_0,r} \cp^1$ be the surface obtained from $\SR$ by multiplying the metric in $B_r(z_0)$ by $\tfrac{1+|z|^2}{r^2 + |z|^2/r^2}$. Then $g_r : \SR \sharp_{z_0,r} \cp^1$ is such that $|\dd g_r| \leq 10 |\dd g|$ on $B_{r(1+r^\eps)}(z_0)$ and $\db_J g_r =0$ on $B_r(z_0)$.
\end{lema}
\begin{proof}
To achieve a bounded differential on $B_r(z_0)$ the metric has to be changed: $\dd \habr$ can only be expected to be bounded for the metric as introduced in \cite[\S{}2.2]{moi-cyl} (see also \cite[\S{}10.3]{mds1}). Morally, this comes from the fact that $\habr$ will send a disc of radius $r$ to (almost all) the image of some fixed (depending on $b$) $J$-holomorphic map $\cp^1 \to M$; in short the differential (with the standard metric) is expected to be big on this disc. The conformal change of metric gives to a disc of radius $r$ in $\cp^1$ the metric pulled-back from the Fubini-Study metric on $\cp^1$, by the map $\cp^1 \to \cp^1: z \mapsto r^2/z$. This conformal change of metric will ensure that the map has bounded differential (by a constant which depends linearly on $|a|$ and $|b|$, thanks to the compactness of $M$) on $B_r(z_0)$. As for the region $B_{r(1+r^\eps)}(z_0) \setminus B_r(z_0)$:
\begin{equation*}
\begin{array}{rl}
\abs{\dd g_r (z)} &\leq \abs{\dd g(z)} +\babs{(g(z)-\habr(z)) \dd \beta(z)} +\babs{\dd \habr(z)} \\
                &\leq \abs{\dd g(z)} +\babs{b(\frac{\abs{z}^2-r^2}{z} +O(r^{1+\eps_0}) ) \frac{1}{\abs{z} \ln (1+r^\eps) } } +\babs{\dd \habr(z)} \\
              &\leq \abs{\dd g(z)} +\abs{b}\frac{\abs{\abs{z}^2-r^2}}{\abs{z}^2 \ln (1+r^\eps) } + \frac{O(r^{1+\eps_0})}{\abs{z} \ln (1+r^\eps) } +\babs{\dd \habr(z)} \\
            &\leq \abs{\dd g(z)} +\abs{b}\frac{(2r^\eps+r^{2\eps})r^2 }{r^2 \ln (1+r^\eps) } + O(r^{\eps_0-\eps}) +\babs{\dd \habr(z)} \\
            &\leq 2 \abs{\dd g(z)} + O(r^{\eps_0-\eps}) +\babs{\dd \habr(z)}. \\
\end{array}
\end{equation*}
So that $\|\dd g_r\|_{\linf} \leq 10 \|\dd g \|_{\linf}$. 
\end{proof}
\par Upon reading Donaldson's method in \cite[\S{}3]{don} one might think that here the introduction of the function $\habr$ is superfluous. Indeed, in the cited paper, it seems sufficient to modify first the term in $b \bar{z}$ on a thin annulus, and then cut-off completely the term in $az$ on a larger annulus. However, this process does not apply here as the addition is only defined in a chart. The truncation of the $az$ term would have to be made for $z$ of small norms, and there would be no guarantee that the holomorphic function which  substitutes to $b \bar{z}$ would not go out of the chart.
\par A priori, this grafting process only makes the $\db_J$ trivial in a neighborhood of the point where a grafting occurred. For a more global decrease of the $\db_J$, our candidate $f$ to a implicit function theorem will be obtained by repeating this process. The construction of this $f$ (so that the Taubes' norm of $\db_J f$ is small) can now be described.
\par First take a $R_0$ so that for every parameters in the local expansions of $g$ the construction of \cite[Theorem 1.3]{moi-cyl} works when $r<R_0$. Let $S_0 \subset \SR$ be the set of points where $\db_J g \neq 0$. We want to cover $S_0$ with discs so that the geometry (curvature of $\SR$ and the boundary of $S_0$) will make only a small perturbation. Take this radius $R_1<R_0$ so that furthermore $R_1< 10^{-10/\eps}$ (\ie $(1+R_1^\eps)$ is less than $1+10^{-10}$). First pick a set $\Omega_1$ of densely packed discs of radius $r_1(1+r_1^\eps)$ where $r_1 < R_1$, and let $S_1 = S_0 \setminus \cup_{z \in \Omega_1} B_{r_1(1+r_1^\eps)}(z)$. Then a second set $\Omega_2$ of points $z \in S_1$ where one can put the disc of biggest radius (the radius $r_{2,z}$ depending this time on the point), and set $S_2 = S_1 \setminus \cup_{z \in \Omega_2} B_{r_{2,z}(1+r_{2,z}^\eps)} (z)$. Continue this process $N_s$ times to get $S:=S_{N_s}$ with $vol(S_{N_s}) \leq 2^{-N_s} vol(\SR)$, $\Omega = \cup_{i=1}^{N_s} \Omega_i$ and $5^{-i+1} r_1 \leq r_{i,z} \leq 7^{-i+1} r_1$. Then at every point of $z \in \Omega_i$ make a $J$-holomorphic graft of parameter $r_{i,z}$ (the grafted function is $H^{a_z,b_z}_{r_{i,z}}$ where $a_z$ and $b_z$ are the holomorphic and anti-holomorphic coefficients of $g$ at $z$).

\par The approximate solution $f$ will then be characterized by the following information: the radius $r_{i,z}$ of each grafting operation, the volume of the region $S$ where no grafting occurred. The number of grafting one makes is not bounded if one tries to make $S$ as small as possible. Let us also introduce $r = \max r_{i,z} = r_1$ the biggest radius for which this surgery is done, $\rmxi = \max_{z \in \Omega_i} r_{i,z}$ the biggest radius at the $i$\up{th} step, $\rmin = \min r_{i,z}$ the smallest radius, and $\lambda = r /\rmin \leq 7^{N_s}$. 
\par Checking that the Taubes' norm of $\db_J f$ is small is a relatively simple computation. 
\begin{lema} \label{bornas}
Let $f$ be as above. Suppose $N_s \geq -\ln (10 \rho^2 r_1^\eps) / \ln 2$ and $\rho > r_1 (=r)$ then $\nr{\db_J f}_{*, \rho} \leq c_{13} \rho^2 r^\eps \ln(\rho^2 r^\eps)$
\end{lema}
\begin{proof}
The assumption on $N_s$ is made so that $S_{N_s} \leq 5 vol(\SR) \rho^2 r_1^\eps$. The desired quantity is
\[\eqtag \label{nrdbju}
\|\db_J f\|_{*,\rho} = \sup_{z \in \SR} \int_{B_\rho(z)} |\ln d(z,y)| |\db_J f(y)| \dd vol(y).
\]
The ball of radius $\rho$ will thus encounter many regions where a grafting has been done. More precisely, in a ball of radius $\rho$ there will be less than $10 vol(\SR) \rho^2/r_1^2$ balls of radius of the first step ($N_s=1$), then for every following step less than $4$ times the number of the previous step. In short the area of all the annuli inside a ball of radius $\rho >r_1$ is less than 
\[
10 vol(\SR) \frac{\rho^2}{r_1^2} ( \sum_{i\geq 1} 4^{i-1} \rmxi^{2+\eps} ) 
  \leq 10 vol(\SR) \rho^2 r_1^\eps ( \sum_{i\geq 0} 4^i 5^{-(2+\eps)i} )
  \leq  20 vol(\SR) \rho^2 r_1^\eps.
\] 
Split the integral in \ref{nrdbju} between a ball of radius 
\[
\rho_0 = \sqrt{ vol(S_{N_s}) + 20 vol(\SR) \rho^2 r_1^\eps} \leq 5 vol(\SR)^{1/2} \rho r_1^{\eps/2} \leq \rho,
\]
where the integration of the singular ($\ln$) kernel will take place, and the annulus $B_\rho \setminus B_{\rho_0}$ to obtain:
\[
\begin{array}{rcl}
\|\db_J f\|_{*,\rho} 
  &\leq & 20 \pi \|\dd g\|_{\linf} \big(2 \rho_0^2 |\ln \rho_0| + (vol(S_{N_s}) + 20 vol(\SR) \rho^2 r_1^\eps)|\ln \rho_0| \big) \\
  &\leq & 750 \pi \|\dd g\|_{\linf} vol(\SR) \rho^2 r_1^\eps  |\ln 25 vol(\SR) \rho^2 r_1^\eps|. \qedhere
\end{array}
\]
\end{proof}
\par Before moving on, a somehow intermediate function $f_0$ between the initial function $g$ (obtained by a $\cinf$ extension of $g_0$) and the approximate solution resulting from the grafting process $f$ must be introduced. Intuitively, it looks as if we removed all the grafts from $f$, leaving stubs where they used to stand. This intermediate function will be needed as the analysis will be split between the part on $\SR$ and that on the graft. In local charts around grafting points $f_0$ is defined as follows: 
\[
f_0(z) = \left \{ \begin{array}{llrcl}
  g(z)                                             & \textrm{si} &r(1+r^\eps)      <& \abs{z} & \\
  \beta(\abs{z})g(z) + (1-\beta(\abs{z})) \habr(z) & \textrm{si} &r               <& \abs{z} &< r(1+r^\eps) \\
  \beta(|\frac{r^2}{z}|)g(z) + (1-\beta(\abs{\frac{r^2}{z}})) \habr(z)        
                                                   & \textrm{si} &r(1+r^\eps)^{-1} <& \abs{z} &<r           \\
  g(z)                                             & \textrm{si} &                 & \abs{z} &<r(1+r^\eps)^{-1}.
\end{array} \right.
\]

\subsection{Splitting the linear equation} \label{Ts6}

So far, an approximate solution $g$ has been produced and, in order to keep its differential small, conformal changes of metric must be operated on the surface $\SR$. Let $\SR_\Omega= \SR \sharp_\Omega \cp^1$ denote this surface endowed with a new metric; there is no control on the number of surgeries $|\Omega|$ and consequently on the volume of this manifold. Actually, even the injectivity radius can only be bounded from below by $\rmin$. Given that the estimates of section \ref{chanell} are done for a manifold of fixed volume and injectivity radius, these methods will deal with the linear equation on the whole of $\SR_\Omega$. Instead, the problem will be split between the initial surface $\SR$ (together with the intermediate function $\tilde{g}$) and the $|\Omega|$ grafts of $\cp^1$ (together with the localised solutions $\habr$), with some compatibility conditions. The inversion of the linear operator (or equivalently the resolution of these linear equations) will be dealt with in the next subsection whereas the non-linear equation is discussed in subsection §\ref{Ts8}.
\par Consider the open covering $\{ U_i\}_{0 \leq i \leq |\Omega|}$ of $\SR$ defined as follows. For $i>0$, each $U_i$ is the interior of the holomorphic part of the grafts: each $U_i$ is a ball of radius $r_{l,z_i}$ around $z_i \in \Omega_l$. Let $\phi_i:U_i \inj \cp^1$ be the identification of that disk to a disk of the same radius in $\cp^1$; recall that the metric on this region can be identified to that of the complement (by inversion) of a disc of radius $r_i$ in $\cp^1$. Still for $i>0$, let $f_i:\cp^1 \to M$ denote the that was grafted on this disc; more precisely, $f_i = H_{a,b}^{r_{l,z_i}}$ where $a = \del_{J(z)} f(z)$ and $b = \db_{J(z)} f(z)$. Then $f \bv_{U_i} = f_i \circ \phi_i $ by construction. 
Now let $U_0$ be the open set obtained by thickening $\SR' := \SR \setminus \cup_{0<i \leq |\Omega|} U_i$. Note that $f\bv_{\SR'} = f_0 \bv_{\SR'} $, but $f$ and $f_0$ are different on $U_0 \setminus \SR'$. Thus, the functions $\phi_i^*$ are the identity when $i>0$, and $\phi_0^*$ is the identity on $\SR'$ (but not on the intersections $U_0 \cap U_i$). To avoid confusion, these functions will always be written in the notation. 
\par Let $[\eta] =\{\eta_i \in \cinf(\SR_i, \Lambda^{0,1} f_i^*\tg M ) \}$ be given on $\SR_0=\SR$ and each grafted $\SR_i=\cp^1$ ($0< i \leq |\Omega|$), and let $[\xi] = \{\xi_i \in \cinf(\SR_i, f_i^*\tg M) \}$. Proper relations between those quantities must be chosen so that a solution to all $D_{f_i} \xi_i = \eta_i$ allows the construction of a solution for $D_f \xi = \eta$. A naive train of thought would have that from a given $\eta$, the $\eta_i$ could be constructed so that when the equations $D_{f_i} \xi_i = \eta_i$ are solved, a $\xi$ can be directly constructed. Unfortunately, a slightly more involved procedure has to be done. In particular, the $\eta_i$ will depend linearly on the $\xi_i$; an existence (and estimates on the norms) of solutions can only be made for a certain choice of parameters. 

\par On $\phi_i(U_i \setminus U_0)$ (where $i>0$) the relations are simply $\phi_i^* \xi = \xi_i$ and $\phi_i^* \eta = \eta_i$. Similarly, on $\phi_0(U_0 \setminus \cup_{0<i \leq |\Omega|} U_i)$, $\phi_0^* \xi = \xi_0$ and $\phi_0^* \eta = \eta_0$. The regions requiring more care are $U_0 \cap U_i$. Identify this region to the annulus $A_{r(1+r^\eps)^{-1},r}$, let $\mu \geq \eps$ and let $s_i$ be such that $r(1+r^\mu)^{-1} = s_0 < s_1 < s_2 < s_3 = r$, and let $\gamma_1$ and $\gamma_2$ two cut-off functions such that: $\gamma_1(|z|) = 0$ if $|z|\geq s_3$ and $\gamma_1(|z|) = 1$ if $|z| \leq s_2 $, while $\gamma_2(|z|)=0$ if $|z| \leq s_0$ and $\gamma_2(|z|)=1$ if $|z| \geq s_1$. Consequently, let
\[
\xi = \gamma_1 \phi_i^{-1*} \xi_i + \gamma_2 \phi_0^{-1*} \xi_0.
\]
Furthermore, on $\phi_i(\{|z| < s_2\})$
\[
\eta_i = \phi_i^* \eta - D_{f_i} \phi_i^* \phi_0^{-1*} (\gamma_2 \xi_0),
\]
but when $|z| > s_1$, that is on $\phi_0(\{ |z| > s_1 \}) $,
\[
\eta_0 = P_{0,1}^{f_0} [ \phi_0^* \eta - \phi_0^* \phi_i^{-1*} D_{f_i} (\gamma_1 \xi_i) - (\phi_0^* D_f \phi_0^{-1*} - D_{f_0}) \xi_0].
\]
The values of $\eta_i$ on $\phi_i(\{|z| > s_2\})$ and those of $\eta_0$ on $\phi_0(\{ |z| < s_1 \})$ are not relevant. The projection is present to make sure that the forms are of the correct type, since the transport by $\phi_0^*$ need not preserve the forms of type $(0,1)$. The next lemma justifies that this projection will not jeopardize the construction, given the perturbation is not too big.
\begin{lema}\label{dproj}
Let $J$ be an almost-complex structure on $\cc^n$. Then there is a constant $c$ with the following property. Let $U$ is an open set of the complex plane $(\rr^2,j)$, let $f_1,f_2: U \to (\cc^n,J)$, let $\eta \in \Lambda^1 \otimes f_1^*\tg \cc^n$, and let $\Phi: f_1^*\tg \cc^n \to f_2^*\tg \cc^n$ is given by parallel transport. If $P_{1,0}^{f_1} \eta =0$, $P_{0,1}^{f_2} \Phi \eta =0$ and  $\nr{f_1-f_2}_{\linf(U)} < c$ then $\eta=0$.
\end{lema}
\begin{proof}
When the two functions are close enough (depending on the structure $J$), the parallel transport is made along small paths so that the anti-holomorphic part and holomorphic remain linearly independent. Obviously, it would possible that, for two point sufficiently far apart $m_1$ and $m_2$, parallel transport of $J_{m_1}$ to the point $m_2$ gives $-J_{m_2}$.
\end{proof}
\par The obstructions present to to solve $D_f \xi = \eta$ have of course not disappeared by rewriting the equation in this local form. However, on the $\cp^1$ an inverse for $D_{f_i}$ exist, so the case $i>0$ will have a different treatment from the case $i=0$.

\begin{lema}
A local solution $[D_{f_i}\xi_i] = [\eta_i]$ gives a solution to $D_f\xi = \eta$.
\end{lema}
\begin{proof}
Indeed, when $|z| < s_2$ then, using $D_f  \phi_i^{-1*} =  \phi_i^{-1*} D_{f_i}$,
\[
\begin{array}{rcl}
D_f \xi 
   & = & \phi_i^{-1*} D_{f_i} \xi_i + D_f \phi_0^{-1*} (\gamma_2 \xi_0) \\
   & = & \eta - \phi_i^{-1*} D_{f_i} \phi_i^{*} \phi_0^{-1*} (\gamma_2 \xi_0) + D_f \phi_0^{-1*} (\gamma_2 \xi_0)\\
   & = & \eta
\end{array}
\]
As for when $|z|>s_1$,
\[
\begin{array}{rcl}
\phi_0^{*} D_f \xi 
   & = & \phi_0^{*} D_f (\gamma_1 \phi_i^{-1*} \xi_i) + D_{f_0} \xi_0 + (\phi_0^{*} D_f \phi_0^{-1*} - D_{f_0}) \xi_0 \\
   & = & \phi_0^{*} D_f (\gamma_1 \phi_i^{-1*} \xi_i) + P_{0,1}^{f_0} \phi_0^* \eta - P_{0,1}^{f_0} \phi_0^* \phi_i^{-1*} D_{f_i} (\gamma_1 \xi_i) - P_{0,1}^{f_0} (\phi_0^* D_f \phi_0^{-1*} - D_{f_0} \xi_0) \\
  &   & \quad + (\phi_0^{*} D_f \phi_0^{-1*} - D_{f_0}) \xi_0 \\
  & = & P_{0,1}^{f_0} \phi_0^{*} \eta + P_{1,0}^{f_0} ( \phi_0^{*} D_f (\gamma_1 \phi_i^{-1*} \xi_i) + P_{1,0}^{f_0} (\phi_0^{*} D_f \phi_0^{-1*} - D_{f_0}) \xi_0 ). \\
\end{array}
\]
Thus $P_{0,1}^{f_0} \phi_0^* (D_f \xi -\eta) =0$. As $P_{1,0}^f (D_f \xi - \eta) = 0$, we conclude that $D_f \xi = \eta$ using lemma \ref{dproj}.
\end{proof}

\subsection{Solving the linear equation} \label{Ts7}

\par In this section, we are thus looking to find $\xi_i$ the solution of an equation which depends on two parameters, $\xi_0$ and $\eta_i$, \ie for sections $\xi_0 \in \cinf(f^*\tg M)$ and $\eta \in \cinf(\Lambda^{0,1}\otimes f^* \tg M)$ given, the equations $D_{f_i} \xi_i = \eta_i$ will determine $h_i$, for $i>0$, with a linear dependance on $\xi_0$ and $\eta_i$. This done, $\xi_0 $ will be expressed as the solution to a linear equation, with an non-homogenous term in $\eta_0 $ and $\eta_i$. However, to get back to second-order equations, write, for $i\geq 0$, $\xi_i = D_{f_i}^* h_i \in \cinf(f^*\tg M)$. The equation on $\SR $ will be solved using results from section \ref{Ts5}, in particular $E$ will be assumed sufficiently big. The desired $h_0$ is a fixed point of 
\[ \label{ptfxs7}
h_0 = H_0(\eta_0 (h_0, \eta_{i\geq 0})).
\]
The main result is to construct a multilinear map
\[
H_0: \cinf(\Lambda^{0,1} \otimes f_0^*\tg M \bv_{U_0}) \times_{i=1}^{|\Omega|} \cinf(\Lambda^{0,1} \otimes f_i^*\tg M \bv_{U_i}) \to \cinf(\Lambda^{0,1} \otimes f_0^*\tg M)
\]
with the property that $h_0 = H_0(\eta_0,\eta_{i>0}) $ is a fixed point of \eqref{ptfxs7}, and that its norm $\jo{L}$ is bounded by $\gen{\eta_0}_\rho + \sup_i \gen{\eta_i}_{10^{-1}} $. In other words, each $h_i$ is a solution to the equations 
\[
\begin{array}{rll}
D_{f_i} D_{f_i}^* h_i &=  \eta_i =  \phi_i^*\eta                        
                       & \textrm{on } \phi_i(U_i \setminus U_0),\\
D_{f_i} D_{f_i}^* h_i &= \eta_i =  \phi_i^*\eta - D_{f_i} D_{f_i}^* \phi_i^* \phi_0^{-1*}(\gamma_2 h_0) 
                       & \textrm{on } \phi_i(U_i \cap U_0).
\end{array}
\]
Theorem \ref{exborn} and corollary \ref{Tt5.7} will be used to obtain the first bounds. The following estimates related to the gluing functions will play a role in those bounds.
\begin{lema} \label{dg}
Let $\gamma_i$ be the cut-off functions described above, and suppose $r < e^{-10}$. Then
  \begin{enumerate} \renewcommand{\labelenumi}{{\normalfont \alph{enumi}.}}
  \item $ \| \un_{\mathrm{supp} \nabla \gamma_i} \|_{*,\rho} = 4 r^{2+\mu_i} |\ln r|$
  \item $\| \un_{\mathrm{supp} \nabla \gamma_i} \|_{2*,\rho} = 4 r^{1+\mu_i/2}|\ln r|^{1/2}$
  \item $\| \nabla \gamma_i \|_{*,\rho} \leq 4 r|\ln r|$. 
  \item $\| \nabla \gamma_i \|_{2*,\rho} \leq 4 r^{-\mu_i/2}|\ln r|^{1/2}$.
\end{enumerate}
\end{lema} 
\begin{proof}
(a) and (b) derive from relatively straightforward estimates on the area of the gluing annuli. (c) and (d) are a consequence of these two and the fact that  $\nr{\nabla \gamma_i}_{\linf} \leq r^{-1-\mu_i}$. 
\end{proof}
Shorten the notation $\gen{\, \cdot \,}$ of \eqref{croch} further by
\[ \eqtag \label{delt}
\begin{array}{lcl}
\Delta_{0,\rho} & = & \pgen{P_{0,1}^{f_0} \phi_0^*\eta}_\rho, \\
\Delta_i & = & \pgen{\phi_i^*\eta}_{10^{-1}}.
\end{array}
\]
For now, small eigenvalues of the Laplacian are not of concern to us. Theorem \ref{exborn} gives a solution of the form $h_0 = H_0(\eta_0)$ to $D_{f_0} D_{f_0}^* h_0 = \eta_0$, where $H_0$ depends linearly on $\eta_0$. The form of the desired estimate is now:
\[
\nr{h_0}_{\jo{L},\rho} \leq K ( \Delta_{0,\rho} + \supp{i} \Delta_i).
\]
However, $\eta_0$ is not only function of $\eta$ but also of the $h_i$ for $i>0$. Furthermore, the $h_i$ depend on $h_0$. Finding a fixed point for  $h_0 = H_0(\eta_0(\eta,h_0))$ can be perceived as the following process: first, the $h_i$ are determined for $i>0$, using  $h_0=0$. At the next step, $h_0$ is found for these $h_i$ ($i>0$), which are then computed anew for this $h_0$. If the process is contracting, then it converges to the desired fixed point. We will not go to and fro between $i=0$ and $i>0$ explicitly, but we shall show that $H_0$ as a function of $h_0$ has a fixed point. 
\begin{lema}\label{Tt7.2}
Let $\{h_i\}$ and $\{\eta_i\}$ be as above, then, for $i>0$,
\[
\nr{h_i}_{\jo{L},10^{-1}} \leq c_{14} ( \Delta_i + |\ln r|  \nr{h_0}_{\jo{L},\rho} ).
\]
Furthermore, on $U_0 \cap U_i$, 
$|\nabla h_i| \leq  c_{14} \nr{h_i}_{\jo{L},10^{-1}}.$
\end{lema}
\begin{proof}
 As $h_i$ depends linearly on $\eta_i$ we will bound its norm according to a decomposition of $\eta_i$ in two terms. Write $\eta_i = \eta_i^{(1)} - \eta_i^{(2)}$ where
$\eta_i^{(1)} =  \phi_i^* \eta$ and
\[
\begin{array}{rcl}
\eta_i^{(2)} 
 & = &D_{f_i}\phi_i^*\phi_0^{-1}(\gamma_2  D_{f_i}^* h_0)\\
 & = & (\sigma_i(\nabla) + l_i) \gamma_2 (\sigma_i^*(\nabla) + l_i^*) h_0\\
 & = & \sigma_i(\nabla) \gamma_2 \sigma_i^*(\nabla) h_0 + k_1 \nabla (\gamma_2 h_0) + k_2 h_0\\
 & = & k_0 \nabla(\gamma_2 \otimes \nabla h_0)  + k_1 \nabla (\gamma_2 h_0) + k_2 h_0 \\
\end{array}
\]
for appropriate symbols and tensors such that $|k_j|<K_1$. Accordingly, split $h_i = h_i^{(1)} + h_i^{(2)}$. The part $h_i^{(1)}$ of $h_i$ coming from $\eta_i^{(1)}$ is bounded by $\Delta_i$ according to corollary \ref{Tt5.7}. 
\par As for the part of the norm coming from $h_i^{(2)}$, the solution of $D_{f_i} D_{f_i}^* h_i^{(2)} = \eta_i^{(2)}$, use 
\[
\begin{array}{lcl}
q   &=& 
    k_1 \nabla (\gamma_2 h_0) + 
    k_2 h_0 \\
b_1 &=& k_0 \\
b_2 &=& \gamma_2 \otimes \nabla h_0
 \end{array}
\]
Easily one has $\nr{b_1}_{\jo{L}^0,10^{-1}} \leq K_2$ 
 and 
\[
\begin{array}{lcl}
\nr{b_2}_{\jo{H},10^{-1}} 
  & \leq & K_3 \nr{\nabla h_0}_{\jo{H},\rho}  \\
 \end{array}
\]
As for $\nr{q}_{*,10^{-1}}$, it is bounded by
\[
\begin{array}{l}
K_4 \Big( \nr{h_0}_{\linf} |\ln r| r +  r^{1+\mu_2/2} \nr{\nabla h_0}_{2*,\rho} \Big),
\end{array}
\]
using proposition \ref{l0ssmul}.c and lemma \ref{dg}. This done we now remark that $\nr{h_i}_{\linf}$ is bounded by the same quantity above as $\nr{h_i}_{\linf} \leq \nr{h_i}_{\jo{L},\rho}$ for any $\rho$. Since $h_i$ satisfies a second-order elliptic linear equation with sufficiently regular coefficient, the Cauchy-G{\aa}rding inequality implies $|\nabla h_i|$ is bounded by $K_5 \nr{h_i}_{\linf}$ on the complement of $r(1+r^\mu)^{-1}>|z|$.
\end{proof}
\par It is now time to make some estimates for the solutions of the equation $\Pi_E D_{f_0}^* D_{f_0} h_0 = \Pi_E \eta_0 $. This time we will split $\eta_0$ in three parts:
\[
\begin{array}{lcl}
\eta_0^{(1)}   & = &  P_{0,1}^{f_0} \phi_0^* \eta, \\
\eta_0^{(2)}  & = &  P_{0,1}^{f_0} \phi_0^* \phi_i^{-1*}D_{f_i}(\gamma_1 D_{f_0}^* h_i), \\
\eta_0^{(3)} & = & P_{0,1}^{f_0} (\phi_0^* D_f \phi_0^{-1*}-D_{f_0}) D_{f_0}^* h_0,
\end{array}
\]
where $\eta_0^{(2)}$ has support on $A_{r,r(1+r^{\mu})}$ and $\eta_0^{(3)}$ has support on $A_{r,r(1+r^{\eps})}$. Each of these give rise to $h_0^{(k)}$ solution of $\Pi_E D_{f_0}^* D_{f_0} h_0^{(k)} = \Pi_E \eta_0^{(k)}$.
\begin{lema}\label{Tt7.8}
In the notations above
\[
\begin{array}{lcl}
\pnr{h_0^{(1)}}_{\jo{L},\rho}   & \leq & 
     c_{12} \kappa_1(\rho,E) \Delta_{0,\rho}, \\
\pnr{h_0^{(2)}}_{\jo{L},\rho}   & \leq & 
     c_{15} \kappa_1(\rho,E) r |\ln r| \supp{i > 0} \nr{h_i}_{\jo{L},10^{-1}}, \\
\pnr{h_0^{(3)}}_{\jo{L},\rho}   & \leq & 
     c_{15} \kappa_1(\rho,E) r^{1+\eps} \nr{h_0}_{\jo{L},\rho}.
\end{array}
\]
\end{lema}
\begin{proof}
The first one, $h^{(1)}_0$, is bounded directly using theorem \ref{exborn}:
\[
\nr{h_0}_{\jo{L},\rho} \leq c_{12}  \kappa_1(\rho,E) \Delta_{0,\rho}.
\]
As for $h_0^{(2)}$, first we write $D_{f_i} = \sigma_i(\nabla) + l_i$. Then
\[
\begin{array}{rcl}
\eta_0^{(2)} 
  & = & \sigma_0^*(\nabla) P_{0,1}^{f_0} \gamma_1 \sigma_i(\nabla) h_i + k_1 \gamma_1 \nabla h_i + [k_2(\nabla \gamma_1) +k_3] h_i  \\ 
 \end{array}
\]
again for tensors $k_i$. To bound $\pgen{\eta_0^{(2)}}_\rho$, we shall use $b_1 = k_0$, $b_2 = \gamma_1 \sigma_i(\nabla)  h_i$ and $q$ to be the remaining terms. Then $\nr{b_1}_{\jo{L}^0,\rho} \leq K_2$ and, using lemma \ref{Tt7.2} to bound $|\nabla h_i|$,
\[
\nr{b_2}_{\jo{H},\rho} \leq K_3 \nr{\nabla h_i}_{\linf} (\nr{1}_{2*,\rho} + \nr{1}_{\jo{L}^1,\rho} ) \leq K_3 c_{14} r^{1+\mu/2} \nr{h_i}_{\jo{L},10^{-1}}.
\]
Lemmas \ref{dg} and \ref{Tt7.2} also give us the following bound for $\nr{q}_{*,\rho}$:
\[
K_4 \Big( 
r^{1+\mu} |\ln r| \nr{h_i}_{\linf} + 
r |\ln r| \nr{h_i}_{\linf} 
 \Big) .
\]
So
\[
\pgen{\eta_0^{(2)}}_\rho 
  \leq c_{15} 
    ( r + r^{1+\mu/2} )
    \nr{h_i}_{\jo{L},10^{-1}}
\]
Finally, $h_0^{(3)}$ will be bounded using the fact that $f$ and $f_0$ differ only on the gluing region and that this difference is small: if $m=f-f_0$ in appropriate charts, then $|m| < K_5 |z|^{1+\eps}$ and $|\nabla m| < K_5 |z|^\eps$. That said, rewrite
\[
\eta_0^{(3)} = k_0(m) (m, \nabla^{\otimes 2} h_0) + k_1(m) (\nabla m,\nabla h_0) + k_2(m) (m, \nabla h_0) + k_3(m) (\nabla m, h_0).
\]
We wish to bound $\|h_0^{(3)}\|_{\jo{L},\rho}$ by bounding $\pgen{\eta_0^{(3)}}_\rho$. Thus, take $b_1 = k_0(m) m$, $b_2 = \nabla h_0$ and $q$ to be the remaining terms. Then 
\[
\begin{array}{rcl}
\| b_1\|_{\jo{L}^0,\rho} 
  & \leq & K_6 r^{1+\eps} |\ln \rho|^{1/2}, \\
\| b_2\|_{\jo{H},\rho}   
  & \leq & \nr{\nabla h_0}_{\jo{H},\rho},\\
\| q\|_{*,\rho}   
  & \leq & K_7 \big( \|\nabla m\|_{2*,\rho} \|\nabla h_0\|_{2*,\rho} + \| m\|_{2*,\rho}
\|\nabla h_0\|_{2*,\rho} + \|\nabla m\|_{*,\rho} \|h_0\|_{\linf} \big)\\
  & \leq & K_7 r^{1+2\eps} |\ln r| \nr{h_0}_{\jo{L}^0,\rho}.
\end{array}
\]
So
\[
\pnr{h_0^{(3)}}_{\jo{L},\rho} \leq c_{15}  \kappa_1(\rho,E) r^{1+\eps} |\ln r| \nr{h_0}_{\jo{L},\rho}. \qedhere
\] 
\end{proof}
\begin{rema}\label{rempe}
Note that by the exact same methods, estimates are found for $\pnr{\pi_E \eta^{(k)}_0}_{*,\rho}$. Using lemma \ref{Tt4.14}, the same estimates hold up to replacing $\rho$ by $\rho'$ and multiplying the terms by 
\[
\frac{\rho^2 |\ln \rho|}{\rho'^2 |\ln \rho'|}  \kappa_2(\rho',E), \qquad \textrm{where } \kappa_2(\rho',E) = (1+\rho'^2 E^{5/3})
\]
\end{rema}
\begin{rema} \label{lininv}
\par The linear map $H_0(\eta)$ associates to the $\{\eta_i\}_{i\geq 0}$ the solution $h_0$ to $D_{f_0} D_{f_0}^* h_0 = \eta_0$. In our case the $\eta_i$ depend (linearly) on $h_0$ and on $\eta$. In a sense, it can be written as:
\[
H_0(\eta,h_0) = A_1 \eta + A_2 h_0,
\]
for two linear maps $A_i$. Furthermore, the previous lemma gives a bound of the form
\[
\pnr{H_0(\eta,h_0)}_{\jo{L},\rho} \leq a_1 \big( \Delta_{0,\rho} + \sup_{i>0} \Delta_{i} \big)+ a_2 \pnr{h_0}_{\jo{L},\rho},
\]
where the $a_i$ depend on $\rho$ and $r$ (and the $\eta_i$ on the $\phi_i^* \eta$). The crucial point is that $a_2 < 1/2$ for a certain choice of parameters so that $h_0 = H_0(\eta, h_0)$ is contracting in $h_0$ and consequently that $(\Id - A_2)$ has an inverse whose norm ($\jo{L} \to \jo{L}$) is less than $(1-a_2)^{-1}$ (found by expanding $(\Id - A_2)^{-1}$ in power series). Essentially,
\[
\begin{array}{crcl}
     & h_0 & = & A_1 \eta + A_2 h_0 \\
\imp & h_0 & = & (\Id - A_2)^{-1} A_1 \eta \\
\imp & \nr{h_0}_{\jo{L},\rho} & \leq & (1-a_2)^{-1} a_1 \big( \Delta_{0,\rho} + \sup_{i>0}\Delta_{i} \big).
\end{array}
\]
\end{rema}
The next theorem uses all the estimates of this section to realize this plan. \begin{teoa} \label{Tt6.3}
Let $\rho = 10^{-2}\rho' = r^{1/30}$ and $E = r^{-1/40}$. There exists $R_{15} >0$ such that if $r < R_{15}$, then equation $[D_{f_i} D_{f_i}^* h_i]= [\eta_i]$ has a solution satisfying
\[
\| h_0 \|_{\jo{L},\rho} \leq c_{16} Z (\Delta_{0,\rho} + 
   \supp{i > 0} \Delta_i )
\]
where $Z = 10^{10} r^{-7/60} |\ln r|^3$.
\end{teoa}
\begin{proof}
Putting together the results of lemma \ref{Tt7.8}, using \ref{Tt7.2} to estimate $\nr{h_i}_{\jo{L},10^{-1}}$, remark \ref{rempe} to estimate the small eigenvalues, and choosing $\rho' = 10^{-2} \rho$, leads to the following bound for $\nr{h_0}_{\jo{L},\rho}$:
\[
\begin{array}{l}
\kappa_1(\rho,E) |\ln r| \Big[ 
   c_{12} 
      (
\Delta_{0,\rho} + 
        \kappa_2(\rho',E) \Delta_{0,\rho'} )  \\ \qquad \qquad  \qquad +      
   c_{14} c_{15} (1+\kappa_2(\rho',E) ) r |\ln r| \supp{i > 0} \Delta_i  \\ \qquad \qquad \qquad + 
   c_{14} c_{15} r |\ln r|^2 ( \nr{h_0}_{\jo{L},\rho} + \kappa_2(\rho',E) \nr{h_0}_{\jo{L},\rho'}) \\ \qquad \qquad \qquad + 
   c_{15} r^{1+\eps} |\ln r| ( \nr{h_0}_{\jo{L},\rho} + \kappa_2(\rho',E) \nr{h_0}_{\jo{L},\rho'} ) 
\Big]. 
 \end{array}
\]
Let $Z_1$ and $Z_2 \in \rr$ be such that
\[
\begin{array}{rlll}
\kappa_1(\rho,E)  = & 1+ \dfrac{1}{\rho^4 E} & \leq & Z_1, \\
\kappa_2(\rho',E) = & 1+\rho'^2 E^{5/3} & \leq & Z_2. \\
\end{array}
\]
Using $10^{-2} \rho' = \rho = r^A$, $E = r^{-e}$, introducing $Z_i = 10^5 r^{-\zeta_i}$ the above estimate simplifies to 
\[
\begin{array}{rcl}
\|h\|_{\jo{L},\rho}  & \leq & K_1 |\ln r|^3 \big[ 
    r^{-\zeta_1-\zeta_2} \Delta_{0,\rho} 
  + r^{-\zeta_1}(1+r^{-\zeta_2}) r \supp{i > 0} \Delta_i \\
  & & \hspace{6cm} + (1+r^{-\zeta_1-\zeta_2})(r^{1+\eps}+ r) \|h \|_{\jo{L},\rho} \big]
\end{array}
\]
To make use of the argument presented in remark \ref{lininv} (and be coherent with all the other constraints), for some $r$ small enough, the inequalities that need to be satisfied are
\[
\begin{array}{rclcrclcrcl}
2A - \tfrac{5}{3} e + \zeta_2 & \geq & 0 
& ~ & 
2A - \tfrac{5}{3} e & \leq & 0 
& ~ & 
~ 
\\ 
-4A +  e + \zeta_1 & \geq & 0 
& ~ & 
-4A +  e & \leq & 0 
& ~ & 
1 & > & \zeta_1+\zeta_2
\\ 
\zeta_1 & \geq & 0
& ~ & 
\zeta_2 & \geq &  0
& ~ &
\mu & \geq & \eps \in ]0,\tfrac{1}{3}[
\end{array}
\]
The values $\mu_1 = \eps = \tfrac{1}{4}$, $A = \tfrac{1}{30}$, $e= \tfrac{1}{40}$, $\zeta_1 = \tfrac{7}{120}$ and $\zeta_2 = 0$  are among the possible choices.
\end{proof}

\subsection{Contraction mapping and the non-linear equation} \label{Ts8}

The point of this section is to find, ultimately thanks to a fixed point theorem, a solution to the equation
\[ \eqtag \label{eqnl}
P_{0,1}^{f} \db_J (\exp_{f} \xi) - \db_J f = \chi.
\]
\par The passage from the nonlinear equation to the linear equation will be made by writing
\[
D_f D_f^* h = \eta,
\]
where the right-hand side (actually a function of $h$) contains all the non-linear terms: 
\[ \eqtag \label{etanl}
\begin{array}{rcl}
\eta (h) 
   & = & \chi - D_f^* \big( ( P_{0,1}^f \circ J_{\exp_f h} - J_f)\circ \dd \, \circ j \big) h \\
   & = & \chi - k_1(h) h \otimes h - k_2(h) h \otimes \nabla h - k_3(h) \nabla h \otimes \nabla h - k_4(h) h \otimes \nabla^{\otimes 2} h,
\end{array}
\]
where the $k_i$ are some analytic tensors (given that $|h|$ is small enough) depending on the complex structure $J$ and the differential of $f$. They represent quadratic (and higher) terms  in the expansion of $P_{0,1}^{f} \db_J (\exp_{f} \xi) - \db_J f$ in terms of $\xi$; the linear term being naturally the linearized operator $D_f \xi$ (see for example McDuff and Salamon's \cite[proof of proposition 3.5.5]{mds1}). 

\par Let $\tilde{\mathsf{C}}^\infty(\Lambda^{0,1}\otimes f^*\tg M)$ be the Fr{\'e}chet space of the $[h]$ in $\times_{i=0}^{|\Omega|} \cinf(\Lambda^{0,1}\otimes f^* \tg M \bv_{U_i})$ satisfying the compatibility conditions on the intersections of $U_0$ and $U_{i>0}$. Then, let
\[
\tau:  \cinf(\Lambda^{0,1}\otimes f_0^* \tg M \bv_{U_0}) \times_{i=0}^{|\Omega|} \cinf(\Lambda^{0,1}\otimes f_i^* \tg M \bv_{U_i}) \to \tilde{C}^\infty(f^*\tg M),
\]
the map described in subsection \ref{Ts6} sending $[h] = \{ h_i \}_{i\geq 0}$ to a vector field $\xi$ on $f^*\tg M$.
Finally, let $H_E$ be the map described by theorem \ref{Tt6.3} (the dependence on $E$ being of importance).
\begin{dfa}
  A $E$-quasi-solution $[h] \in \tilde{\mathsf{C}}^\infty(\Lambda^{0,1} \otimes f^*\tg M)$ is a fixed point of 
\[
[h] = H_E (\big[ \eta \big( \tau ([h]) \big) \big] ).
\]
\end{dfa}
\begin{rema}
Let $E$ be so that theorem \ref{Tt6.3} applies, let $[\eta]$ depending on $h$ as above, and let $[h]$ be a $E$-quasi-solution. If furthermore,
\[
(1- \Pi_E) (D_{f_0}D_{f_0}^*h_0 - \eta_0) =0
\]
then $h$ is a solution to the non-linear equation \eqref{eqnl}. Most importantly this additional constraint is finite dimensional, so an $E$-quasi-solution fails only in a finite number of ways to solve \eqref{eqnl}.
\end{rema}
That said, this subsection will focus on the existence of $E$-quasi-solutions. For $\rho >0$, a norm is defined on $\tilde{C}^\infty(\Lambda^{0,1} \otimes f^*\tg M)$ by associating to $[h] = \{h_i\}_{i=0}^{|\Omega|}$ the quantity
\[
\nr{[h]}_{*, \rho} = \| h_0 \|_{*,\rho} + \supp{i} \| h_i \|_{*,10^{-1}}
\]
The ball of radius $d$ for this norm will be denoted $\scr{B}_{\rho,d}$.
\begin{teoa} \label{Tt8.4}
 Let $r < R_{16} $, $d < (c_{17} Z)^{-1}$, $Z = 10^5r^{-7/120}|\ln r|^3$, $E = r^{-1/40}$ and $\rho = r^{1/30}$. If $\chi \in \scr{B}_{\rho,d}$, then there exists a $E$-quasi-solution $h$ depending smoothly on $\chi$ with
\[
\| [ h ] \|_{*,\rho} \leq 2 \| \chi \|_{*,\rho}
\]
\end{teoa}
\begin{proof}
To do so the map $[h] \mapsto H_E(\big[\eta\big(\tau([h])\big)\big])$ will be shown to be a contraction mapping on a ball $\scr{B}_{\rho,d}$. Let $h^{(1)}$ and $h^{(2)} \in \tilde{\mathsf{C}}^\infty(\Lambda^{0,1} \otimes f^*\tg M)$, then the desired inequality is 
\[ \eqtag \label{contr}
\| H_E(\eta(h^{(1)} + h^{(2)}))_i - H_E(\eta(h^{(1)}))_i  \|_{\jo{L},\rho} \leq \kappa \|h^{(2)}_i\|_{\jo{L},\rho},
\]
for some $\kappa \leq 1/4$. Express $\tilde{\eta}_i := \eta_i( h^{(1)} + h^{(2)}) - \eta_i(h^{(1)})$ as
\[
\underset{l_1,l_2 \in \nn}{\sum_{0 \leq l_1 + l_2 \leq 2} } k_{i;l_1,l_2}(h^{(1)}_i,h^{(2)}_i) \nabla^{\otimes l_1} h^{(1)}_i \nabla^{\otimes l_2} h^{(2)_i},
\]
where $k_{i;l_1,l_2}$ are the appropriate analytic tensors. To bound $\| H_E(\tilde{\eta}_i) \|_{\jo{L},\rho}$ (the left hand side of \eqref{contr} as $H_E$ is linear) the $\Delta$ from \eqref{delt} must be evaluated. Decompose $\tilde{\eta}_i = q_i + b_{i;1} \nabla b_{i;2}$ and 
\[
q_i = \sum_{l_1, l_2 \in \{0,1\}} k_{i;l_1,l_2}(h^{(1)}_i,h^{(2)}_i) \nabla^{\otimes l_1} h^{(1)}_i \nabla^{\otimes l_2} h^{(2)}_i,
\]
$b_{i;1} = ( k_{i;0,2} h^{(1)}_i , k_{i;2,0} h^{(2)}_i )$ and $b_{i;2} = (\nabla h^{(2)}_i, \nabla h^{(1)}_i)^\dag $.
Theorem \ref{Tt6.3} then gives 
\[
\| H_E(\tilde\eta) \|_{\jo{L},\rho} \leq K_1 Z \|h^{(1)}\|_{\jo{L},\rho} \|h^{(2)}\|_{\jo{L},\rho},
\]
as is readily checked by bounding $\|q_i\|_{*,\rho}$ by
\[
\begin{array}{l}
 K_2 |\ln \rho|  \rho (\rho \|h^{(1)}_i\|_{\linf} \|h^{(2)}_i\|_{\linf} + \|h^{(1)}_i\|_{\linf}   \|\nabla h^{(2)}_i\|_{2*,\rho}+ \|\nabla h^{(1)}_i\|_{2*,\rho} \|h^{(2)}_i\|_{\linf} ) \\
\hspace{8.8cm} + K_1 \|\nabla h^{(1)}_i\|_{2*,\rho} \|\nabla h^{(2)}_i\|_{2*,\rho}.
\end{array}
\]
Thus when $\kappa = K_1 Z \pnr{h^{(1)}}_{\jo{L},\rho} < 1/4$, the map is contracting.
\end{proof}
In order to obtain a $E$-quasi-solution, it suffices to take $\chi = -\db_J f \in \scr{B}_{\rho,d}$ in \eqref{eqnl} (the fact that this $\chi$ belongs to the said ball for the chosen parameters is a consequence of lemma \ref{bornas}). To alleviate notation, call $H^{nl}_E(f)$ the resulting $E$-quasi-solution of the previous theorem for $\chi = - \db_J f$. 

\subsection{Small eigenvalues} \label{Ts9}

It will now be shown how to use the $E$-quasi-solution of the preceding subsection to obtain a authentic solution. This will be achieved by a fixed point theorem in a small ball $\scr{B} \subset \img \pi_E$. Indeed, to an element $\nu$ of $\scr{B}$ will be associated parameters for a supplementary round of surgeries (that still enable application of theorem \ref{Tt8.4}) and give another approximate solution $f_\nu$. Brouwer's fixed point theorem will then be used to find a zero for the map that send elements of $\scr{B}$ to the small eigenvalues of the $E$-quasi-solution $H^{nl}_E (J f_\nu)$.
\par The idea of the rather rough method used here begins by noticing that the bad part of the $E$-quasi-solution resulting from $H^{nl}_E(\chi)$ is already mostly contained in $\chi$ (see lemma \ref{Tt9.6}). Then the idea is to perturb the approximate solution $f$. Grafting pseudo-holomorphic curves, unlike $\srl{\cp}^2$ in Taubes' work \cite[\S{}9]{tau}, does not seem to give sufficient maneuverability. The extra surgeries will here be performed around points $y$ of the set $S$ where $f$ is not pseudo-holomorphic, surgeries which will essentially be non-holomorphic. So for each $\nu_\alpha$ a basis of $\img \pi_E$, there will be two data: $c_\alpha = \pgen{\nu_\alpha, \nu}$ (the $\nu_\alpha$ component of $\nu \in \img \pi_E$) and the $\nu_\alpha$ component of the surgery near a point $y$. The problem turns out then to try to solve the under-determined (given sufficiently many surgeries are made) system:
\[
\left( \begin{array}{l}
c_1 \\
c_2 \\
\vdots \\
c_{N(E)} \\
\end{array} \right) = \left( \begin{array}{cccc}
p_{1,1} & p_{1,2} & \cdots & p_{1,|\Omega_c|} \\
p_{2,1} & p_{2,2} & \cdots & p_{2,|\Omega_c|} \\
\vdots  & \vdots  & \ddots & \vdots \\
p_{N(E),1} & p_{N(E),2} & \cdots & p_{N(E),|\Omega_c|} \\
\end{array}  \right) \left( \begin{array}{l}
\mu_1 \\
\mu_2 \\
\vdots \\
\mu_{|\Omega_c|} \\
\end{array} \right)
\]
where $N(E)$ is the dimension of $\img \pi_E$, $\Omega_c \subset S$ the set of points $y$ where a surgery is made, and $p_{\alpha,j}$ denotes the $\nu_\alpha$ part of the surgery at $y_j \in \Omega_c$ for some fixed parameter. This system is expected to be undetermined as $N(E)$ is up to a constant $E$ (\ie for the choices in theorem \ref{Tt6.3}, $r^{-1/40}$) while $|\Omega_c|$ has a growth of $r^{-2}$. Solving this (again the technique here is very rough, and can probably be improved) essentially yields the map ``small eigenvalues'' $\to$ ``approximate solutions'' (denoted by $\nu \mapsto f_\nu$). Bounds must be found on these solutions so that the new surgeries preserve the status of $f_\nu$ as an approximate solution (see lemma \ref{Tt9.5ii}) and thus enables to come back, through $H^{nl}_E$, to small eigenvalues.
\par Consequently, the onset of this section corners some more properties of the space of small eigenvalues. Let $\nu_\alpha$ be a $\ldeu$-orthonormal basis of $\img \pi_E$, where $1\leq \alpha \leq N(E) \leq c_9(1+E)$ (see lemma \ref{rgpie}).
\begin{lema} \label{Tt9.6}
Let $\nu \in \img \pi_E$ and $\nu_\alpha$ as above, then there exists $c_{18} \in \rr_{>0}$ such that
\[
\begin{array}{lll}
\|\nu \|_{\linf} & \leq & c_{18} E^{7/3} \rho^{-2} |\ln \rho|^{-1} \| \nu \|_{*,\rho} \\
\|\nu \|_{\ldeu} & \leq & c_{18} E^{7/6} \;\: \rho^{-2} |\ln \rho|^{-1} \| \nu \|_{*,\rho}
\end{array}
\]
\end{lema}
\begin{proof}
Write $ \nu = \sum_\alpha c_\alpha \nu_\alpha$ then $\sum_\alpha |c_\alpha| \leq N(E)^{1/2} \| \nu \|_\ldeu$. Invoking lemma \ref{gradvp} yields
\[
\begin{array}{rcl}
\| \nu \|_{\linf} & \leq & \sum_\alpha |c_\alpha| \| \nu_\alpha \|_{\linf} \\
                & \leq & E^{2/3} N(E)^{1/2} \|\nu\|_\ldeu.
\end{array}
\]
Lemma \ref{rgpie} bounds $N(E)$, while $\| \nu \|^2_{\ldeu} \leq \|\nu \|_{\linf} \|\nu\|_\lun \leq K_1 \rho^{-2} |\ln \rho|^{-1} \|\nu \|_{\linf} \|\nu\|_{*,\rho}$ implies
\[
\nr{\nu}_{\linf} ^2 \leq E^{10/3} \nr{\nu}_{\ldeu}^2 \leq E^{10/3} \rho^{-2} |\ln \rho|^{-1}  \|\nu \|_{\linf} \|\nu\|_{*,\rho},
\]
which gives the first estimate upon dividing by $\| \nu \|_{\linf}$. Using the bound on $\linf$ in the former inequality gives the bound for the $\ldeu$ norm.
\end{proof}
An auxiliary set (which has nothing to do with grafting) will now be introduced to discretize the functions; it is constructed as in the proof of lemma \ref{Tt4.14}. Namely, for $n,m \in \nn$ and $d \in \rr_{>0}$, a set $\Omega_I$ is chosen such that
\begin{itemize}
\item $\cup_{x \in \Omega_I} B_{n d'}(x) = \SR $,
\item $B_{d'}(x) \cap B_{d'}(x') = \vide$ if $x \neq x'$ are two points of $\Omega_I$,
\item for any subset $\Omega_I'$ of cardinality greater than $m$ in $\Omega_I$,  $\cap_{x \in \Omega_I'} B_{n d'}(x) = \vide$.
\end{itemize}
Take a partition of unity $\{\psi_x\}_{x \in \Omega_I}$ relative to the covering $\{B_{nd'}(x)\}_{x \in \Omega_I}$ and define
\[
p:\Omega_I \to \rr_{\geq 0} \qquad \textrm{ by } \qquad p(x) = \int_M \psi_x(y) \dd y.
\]
The next lemma establishes how well the values on this set characterize the functions in $\pi_E$.
\begin{lema} \label{Tt9.12}
Let $\nu \in \img \pi_E$, $\nu_\alpha$, $\Omega_I$ and $p(x)$ as above. Then 
\[
\bigg|\nu - \sum_\alpha \Big( \sum_{x \in \Omega_I} p(x) \pgen{\nu,\nu_\alpha}(x) \Big) \nu_\alpha \bigg| \leq c_{19} d' E^{14/3} \rho^{-2} |\ln \rho|^{-1} \| \nu  \|_{*,\rho}
\]
\end{lema}
\begin{proof}
Rewrite (using $\nu = \sum_\alpha c_\alpha \nu_\alpha$ where $c_\alpha = \int \pgen{\nu,\nu_\alpha}$) the left-hand side as
\[
\bigg| \Big( \sum_\alpha \sum_{x \in \Omega_I} \int \psi_x(z) (\pgen{\nu,\nu_\alpha}(z) - \pgen{\nu,\nu_\alpha}(x) ) \dd z \Big) \nu_\alpha \bigg|.
\]
Then estimate using $\|\nu\|_{\linf} \|\nabla \nu_\alpha \|_{\linf} + \| \nabla \nu\|_{\linf} \| \nu_\alpha \|_{\linf} \leq K_1 E^{3} \rho^{-2} |\ln \rho|^{-1}$ (from lemma \ref{gradvp} and lemma \ref{Tt9.6}), the bound on $N(E)$ (see lemma \ref{rgpie}) and (once again for $\|\nu\|_{\linf}$) lemma \ref{gradvp}.
\end{proof}
An important note is that when $d' \gg r$, almost all the balls around the $x \in \Omega_I$ will have a big intersection with $S$ (the set where no surgeries have been done). A value of $d' = r^{1/4}$ will turn out to be suited to our needs (the error in the above approximation, $d' E^{14/3} \rho^2 \leq r^{1/5}$, will be small for $r \ll 1$)

\par More surgeries will have to be done to $f$ on the set $S$. To do so first pick a set $\Omega_c \subset S$. The function obtained from this last step will henceforth be written $f_\nu$ and depends on $\nu \in \img \pi_E$. It is obtained as follows. On balls of radius $(1+\delta) r'_y$ around $y$, the function $f$ will be modified. Let $\mu_y$ be an extra parameter depending on $\nu$. Assume for now that $\mu_y \leq c_{20}(M,J)$ so that the upcoming construction makes sense in local coordinates. Write $f(z) = az + b \zb + O(|z|^2)$ then in local coordinates around each $y$ let 
\[ \eqtag \label{cokchir}
f_\nu(z) = \left\{ \begin{array}{llrcl}
  f(z)                                 & \textrm{if } & r(1+\delta) <& \abs{z} & \\
  f(z) + \beta(\abs{z}) \mu_y \zb  & \textrm{if } &r          <& \abs{z} &< r(1+\delta) \\
  f(z) + \mu_y \zb                     & \textrm{if } &             & \abs{z} &<r.          
  \end{array}
\right.
\]
where $\beta(|z|) = \dfrac{\ln r(1+\delta) - \ln |z|}{\ln (1+\delta)}$ for $r < \abs{z} < r(1+\delta)$, $=1$ when $|z|<r$ and $=0$ if $|z| >r(1+\delta)$.
\par The aim of the next lemma is to show that the part of the $E$-quasi-solution which fails to be an actual solution is essentially  $\db_J f_\nu$.
\begin{lema} \label{Tt9.4}
Let $h_0$ and $\eta_0$ be as in the $E$-quasi-solution $H^{nl}_E(f_\nu)$ from theorem \ref{Tt8.4}, then
\[
\| \pi_E (D_{f_0} D_{f_0}^* h_0 - \eta_0) + \pi_E P_{0,1}^{f_0} \db_J f_\nu \|_{*,\rho} 
   \leq c_{21}
  (\rho |\ln \rho| + 
   \kappa_1(\rho,E) r |\ln r| + 
   Z \| \db_J f_\nu \|_{*,\rho})
    \| \db_J f_\nu \|_{*,\rho}.
\]
\end{lema}
\begin{proof}
The first term $\pi_E D_{f_0} D_{f_0}^* h_0$ is bounded using the fact that $h_0 \in \Pi_E C^\infty$. In particular,
\[
\| \pi_E D_{f_0} D_{f_0}^* h_0 \|_{*,\rho} \leq K_1 |\ln \rho|( \rho \| \nabla h_0 \|_{2*,\rho} + \rho^2 \| h_0 \|_{\linf})
\] 
for constants that depend on the lower order symbols. In turn, this is bounded by $K_1 \rho |\ln \rho| \| \db_J f_\nu \|_{*,\rho}$ (see theorem \ref{Tt8.4}).  As $\eta$ is $\db_J f_\nu$ and the higher order (or non linear) terms in $h$, the remainder will be split in three terms as in lemma \ref{Tt7.8}. A bound for the two last terms ($\eta^{(2)}$ and $\eta^{(3)}$) is given 
by
\[
K_2 \kappa_1(\rho,E) r |\ln r| \| [h] \|_{\jo{L},\rho} \leq 2 K_2 \kappa_1(\rho,E) r |\ln r| \| \db_J f_\nu \|_{*,\rho}
\]
As for the first ($\eta^{(1)}$), it consists (after substraction of $\db_J f_\nu$) only in higher order terms, and the bound is found by taking $h^{(1)} = -h^{(2)}$ in the proof of theorem \ref{Tt8.4}: $K_3 Z \pnr{[h]}_{\jo{L},\rho}^2$.
\end{proof}
Recall from lemma \ref{bornas} that $\| \db_J f \|_{*,\rho} \leq c_{14} r^\eps \rho^2 |\ln (r^\eps \rho^2) | = c_{14} \tfrac{38}{120} r^{38/120} |\ln r|$, and define $W = P_{0,1}^g (\db_J f_\nu - \db_J f)$. In order to measure the contribution $W$ to $\pi_E$, the part coming from each $y$ according to the alteration of $f$ from \eqref{cokchir} has to be evaluated.
\begin{lema}\label{cccontrib}
For $y \in \Omega_c$, the modification of $f$ as in \eqref{cokchir} contributes to $\pi_E W$ by
\[
\sum_{\alpha =1}^{N(E)} \kappa_3(r'_y,\delta) \langle \nu_\alpha(y), \mu_y \rangle \nu_\alpha + R_y
\]
where $\kappa_3(r'_y,\delta) \in \rr$ is some function asymptotic (for small $r'_y$ and $\delta$) to $2 \pi {r'_y}^2 (1+ \delta + \tfrac{\delta}{2\ln(1+\delta)})$ and $R_y$ is supported on $B_{r'_y(1+\delta)}(y)$ and bounded by $c_{22} \mu_y {r'_y}^3 E^{7/3}$.
\end{lema}
\begin{proof}
The alteration \eqref{cokchir} at $y$ is supported in $B_{r(1+\delta)}(y)$. Denote the difference in $\db_J$ coming from a surgery by $\mu'_y(z) = P_{0,1}^g (\db_J f_\nu(z) - \db_J f(z))$ . Note that $\mu'_y(z) = \mu_y +O(r'_y)$ on $B_{r'_y}(y)$ and $\mu'_y = (\beta(|z|)+ \tfrac{1}{2 \ln(1+\delta)} ) \mu_y + O(r'_y)$ on $B_{r'_y(1+\delta)}(y) \setminus B_{r'_y}(y)$. The corresponding effect on the $\nu_\alpha$ component is 
\[
\int_{B_{r'_y(1+\delta)}(y)} \langle \nu_\alpha(z), \mu'_y(z) \rangle \dd z.
\]
However $|\nu_\alpha(z) - \nu_\alpha(y)| < K_1 E^{2/3} r'_y$ according to lemma \ref{gradvp}. The conclusion follows by putting in the error term $R_y$ both this approximation and the $O$ entering in the expression of $\mu'_y$; a bound by $K_2{r'_y}^3 \mu_y (1+ E^{2/3}) N(E) \nr{ \nu_\alpha}_{\linf}$ is straightforward, and the conclusion is obtained again by lemmas \ref{rgpie} and \ref{gradvp}.
\end{proof}
It is now time to describe how the parameters are set so as to obtain a good approximation of $\nu$.
\begin{lema} \label{Tt9.5iii}
There exist a choice of $\mu_y$ such that 
\[
\| \pi_E W - \nu \|_{*,\rho} \leq c_{23}
  \Big( \frac{  d' E^{3} {r'_y}^2  \rho^{-2} }{\min (vol S, {d'}^2) } +
  d' E^{14/3} \rho^{-2} + 
  \frac{ r'_y E^{14/3}  }{\min (vol S, {d'}^2) }  
 \Big)
|\ln \rho| \| \nu \|_{*,\rho}
\]
and $\mu_y \leq c_{23} \min (vol S, {d'}^2)^{-1} \nr{ \nu }_{\linf}$.
\end{lema}
\begin{proof}
The idea is to push the expression of lemma \ref{cccontrib} to meet the discretization of lemma \ref{Tt9.12}, or in other words to make small the difference  
\[
\sum_{\alpha =1}^{N(E)} \sum_{x \in \Omega_I} \Big( p(x) \pgen{\nu(x),\nu_\alpha(x)} - \sum_{y \in \Omega_c} \kappa_3(r'_y,\delta) \Psi_x(y) \langle \nu_\alpha(y), \mu_y \rangle  \Big) \nu_\alpha
\]
For $y \in B_{d'}(x)$, $| \nu_\alpha(x) - \nu_\alpha(y) | \leq 2 d' \|\nabla \nu_\alpha\|_{\linf} \leq K_1 d' E^{2/3}$ using lemma \ref{gradvp}. Up to this approximation, the above difference can be made $0$ by choosing, for each $y \in B_{d'}(x)$, 
\[
\mu_y = K_1 \kappa_3(r'_y,\delta)^{-1} |B_{d'}(x) \cap \Omega_c |^{-1} \nu(x)
\]
where $|B_{d'}(x) \cap \Omega_c| \geq \tfrac{\min(vol S, {d'}^2)}{2{r'_y}^2}$.
\end{proof}
It is also important that, after the new round of surgeries, the estimates allowing the application of theorem \ref{Tt8.4} still hold.  
\begin{lema} \label{Tt9.5ii}
Let $N_s \geq c_{24} \ln r$ so that $10^{-5} r^{2/3} < vol S < r^{2/3}$ Let $\mu_y$ be chosen as in lemma \ref{Tt9.5iii}, then 
\[
\pnr{W}_{*,\rho} \leq c_{25} E^{7/3} \rho^{-2} |\ln \rho|^2 \pnr{\nu}_{*,\rho}.
\]
In particular, if $ \pnr{\nu}_{*,\rho} \leq (c_{17} c_{25})^{-1} \rho^2 |\ln \rho|^{-2} E^{-7/3} Z^{-1}$ then theorem \ref{Tt8.4} applies to $f_\nu$ and $\chi = - \db_J f_\nu$, so that $H^{nl}_E(f_\nu)$ is a $E$-quasi-solution. 
\end{lema}
\begin{proof}
Once the quantities have been set, an upper bound for $\| W \|_{*,\rho}$ is given by
\[
K_1 \min( vol S, \rho^2) |\ln \rho| \tfrac{2{r'_y}^2}{\min(vol S, {d'}^2)}  \kappa_3(r'_y,\delta)^{-1} \|\nu\|_{\linf} \leq K_2 \|\nu\|_{\linf},
\]
since the number of steps $N_s$ was sufficient (so that $vol S < {d'}^2 = r^{1/2}$). Lemma \ref{Tt9.6} yields the conclusion.
\end{proof}
The time is now ripe to show that a fixed point can be found.
\begin{teoa} \label{Tt9.2}
Let $f_\nu$ be obtained as above from $\nu$, then there exists a $\nu$ such that  
$H^{nl}_E(f_\nu)$ is a $E$-quasi-solution and satisfies $\pi_E (D_{f_0} D_{f_0}^* h_0 - \eta_0) = 0$.
\end{teoa}
\begin{proof}
Let $\scr{B}' = \{ \nu \in \img \pi_E | \pnr{\nu}_{*,\rho} \leq (2 c_{17} c_{25})^{-1} \rho^2 E^{-7/3} Z^{-1} r^{1/120} = K_1 r^{23/120} |\ln r|^3 \}$. Look at the map which assigns to $\nu \in \scr{B}'$ the quantity $F(\nu) = \pi_E (D_{f_0} D_{f_0}^* h_0 - \eta_0)$ obtained from $f_\nu$. Putting together lemmas \ref{Tt9.4}, \ref{Tt9.5iii} and \ref{Tt9.5ii}, this map can be written as $F(\nu) = \nu + R(\nu)$ where
\[
\begin{array}{rcl}
\pnr{F(\nu) - \nu}_{*,\rho} = \pnr{R(\nu)}_{*,\rho} 
  & \leq & \pnr{F(\nu) - \pi_E \db_J f_\nu}_{*,\rho} + \pnr{\db_J f}_{*,\rho} + \pnr{\pi_E W - \nu}_{*,\rho} \\
  & \leq & r^{1/120} \pnr{\nu}_{*,\rho} + c_{13} r^{38/120} + c_{23} r^{26/120} \pnr{\nu}_{*,\rho} \\
  & \leq & K_2 ( r^{1/120} \pnr{\nu}_{*,\rho} + r^{38/120}).
\end{array}
\]
So for $r$ small enough, $R$ maps $\scr{B}'$ into itself. Furthermore $\pnr{R(\nu)}_{*,\rho}/\pnr{\nu}_{*,\rho} \leq 1/2$ when $\nu$ is on the boundary of $\scr{B}'$. The Brouwer fixed point theorem implies that $\nu \mapsto - R(\nu)$ has a fixed point, $\nu_0$, and this yields the conclusion as $F(\nu_0)=0$.
\end{proof}

\end{document}